\documentclass[12pt]{amsart}
\usepackage[margin=1.25in]{geometry} 
\usepackage{amsfonts, amssymb}
\usepackage{array, booktabs, caption, longtable, multicol}
\usepackage[all]{xy}
\usepackage{color, colortbl}
\usepackage{aliascnt}
\usepackage[ocgcolorlinks, breaklinks=true, pagebackref, linktocpage,
allcolors=untgreen]{hyperref} 

\numberwithin{equation}{section}

\theoremstyle{plain}
\newtheorem{lemma}{Lemma}[section]

\newaliascnt{proposition}{lemma}
\newtheorem{proposition}[proposition]{Proposition}
\aliascntresetthe{proposition} 

\newaliascnt{theorem}{lemma}
\newtheorem{theorem}[theorem]{Theorem} 
\aliascntresetthe{theorem}
\newtheorem*{theorem*}{Theorem}

\newaliascnt{corollary}{lemma}
\newtheorem{corollary}[corollary]{Corollary} 
\aliascntresetthe{corollary}

\newaliascnt{definition}{lemma}
 
\aliascntresetthe{definition}

\theoremstyle{remark}
\newaliascnt{remark}{lemma}
\newtheorem{remark}[remark]{Remark} 
\aliascntresetthe{remark}

\definecolor{untgreen}{RGB}{5,144,51}
\definecolor{cublue}{RGB}{75,146,219}
\definecolor{cugold}{RGB}{207,184,124}
\definecolor{lightblue}{rgb}{0.8,0.85,1}
\definecolor{grayshade}{gray}{0.85}
\definecolor{shade}{cmyk}{.29,0,1,0} 

\def\equationautorefname~#1\null{(#1)\null}
\def\itemautorefname~#1\null{(#1)\null}
\def\sectionautorefname~#1\null{\S#1\null}
\def\subsectionautorefname~#1\null{\S#1\null}

\makeatletter \renewcommand\subsection{\@startsection{subsection}{2}%
  \z@{1\linespacing\@plus.2\linespacing}{.5\linespacing}%
  {\normalfont\bfseries}} 
\makeatother

\newcommand\CA{{\mathcal A}} 
\newcommand\CB{{\mathcal B}}

\newcommand\CI{{\mathcal I}} 
\newcommand\CL{{\mathcal L}}

\newcommand\CZ{{\mathcal Z}}

\newcommand\fp{{\mathfrak p}}
\newcommand\fq{{\mathfrak q}}

\newcommand\fm{{\mathfrak m}}
\newcommand\fr{{\mathfrak r}}

\newcommand\FH{\mathfrak H}

\newcommand\BBC{{\mathbb C}}

\newcommand\Fix{{\operatorname{Fix}}}

\newcommand\GL{\operatorname{GL}}

\newcommand\lcm{\operatorname{lcm}}

\newcommand\inverse{^{-1}}
\renewcommand\th{{^{\text{th}}}}

\newcommand\id{{id}}
\renewcommand\r{{\operatorname{ref}}}
\newcommand\ram{{\operatorname{ram}}}
\newcommand\rhoXtilde{\widetilde\rho_X}

\begin{document}

\title[Restricting invariants of unitary reflection groups]
	{Restricting invariants of unitary reflection groups}

\author[N. Amend]{Nils Amend}
\address
{Fakult\"at f\"ur Mathematik,
Ruhr-Universit\"at Bochum,
D-44780 Bochum, Germany}
\email{nils.amend@rub.de}
\email{gerhard.roehrle@rub.de}

\author[A. Berardinelli]{Angela Berardinelli} \address{Department of
  Mathematics and Information Technology\\ Mercyhurst University\\ Erie PA,
  USA 16546} 
\email{aberardinelli@mercyhurst.edu} 

\author[J.M. Douglass]{J. Matthew Douglass} \address{Department of
  Mathematics\\ University of North Texas\\ Denton TX, USA 76203}
\email{douglass@unt.edu}

\author[G. R\"ohrle]{Gerhard R\"ohrle}
\dedicatory{To the memory of Robert Steinberg}

\keywords{Reflection arrangements, unitary reflection groups, invariants,
smooth orbit variety}
\subjclass[2010]{Primary 20F55; Secondary 13A50}

\begin{abstract}
  Suppose that $G$ is a finite, unitary reflection group acting on a complex
  vector space $V$ and $X$ is the fixed point subspace of an element of $G$.
  Define $N$ to be the setwise stabilizer of $X$ in $G$, $Z$ to be the
  pointwise stabilizer, and $C=N/Z$. Then restriction defines a homomorphism
  from the algebra of $G$-invariant polynomial functions on $V$ to the
  algebra of $C$-invariant functions on $X$. Extending earlier work by
  Douglass and R\"ohrle for Coxeter groups, we characterize when the
  restriction mapping is surjective for arbitrary unitary reflection groups
  $G$ in terms of the exponents of $G$ and $C$, and their reflection
  arrangements. A consequence of our main result is that the variety of
  $G$-orbits in the $G$-saturation of $X$ is smooth if and only if it is
  normal.
\end{abstract}

\maketitle

\allowdisplaybreaks

\section{Introduction}\label{sec:intro}

Let $G$ be a finite unitary reflection group acting on a finite dimensional
complex vector space $V$ and let $X$ be a subspace of $V$. Define
\[
N_X=\{\, g\in G\mid g(X)=X\,\},
\]
the setwise stabilizer of $X$ in $G$,
\[
Z_X=\{\, g\in G\mid g(x)=x\, \forall x\in X\,\},
\]
the pointwise stabilizer of $X$ in $G$, and
\[
C_X=N_X/Z_X.
\]
Then $C_X$ acts faithfully on $X$. We frequently identify $C_X$ with its
image in $\GL(X)$ and consider $C_X$ as a subgroup of $\GL(X)$. In this
situation, restriction defines a surjective, degree preserving, algebra
homomorphism $\rhoXtilde \colon \BBC[V] \to \BBC[X]$ from the algebra of
polynomial functions on $V$ to the algebra of polynomial functions on $X$
with $\rhoXtilde(\BBC[V]^G) \subseteq \BBC[X]^{C_X}$. Let
\[
\rho_X \colon \BBC[V]^G \to \BBC[X]^{C_X}
\]
be the algebra homomorphism from $G$-invariant polynomials on $V$ to
$C_X$-invariant polynomials on $X$ obtained by restriction. Note that
$\rho_X$ is surjective if and only if every $C_X$-invariant polynomial
function on $X$ extends to a $G$-invariant polynomial function on $V$. The
main result of this paper, \autoref{thm:main}, is an elementary
combinatorial characterization of when the map $\rho_X$ is surjective in the
case when $X$ is the fixed point space of an element (or equivalently a
subgroup) of $G$. This characterization is in terms of (1) the exponents of
$G$ and the exponents of the subgroup of $C_X$ generated by the elements
that act on $X$ as reflections, and (2) the restriction of the arrangement
of $G$ to $X$. In addition, we show that when $\rho_X$ is surjective, the
group $C_X$ always acts on $X$ as a unitary reflection group. In the course
of the proof of \autoref{thm:main} we classify all pairs $(G,X)$ where $G$
is an irreducible unitary reflection group acting faithfully on $V$, $X$ is
the fixed point set of an element of $G$, and $C_X$ acts on $X$ as a unitary
reflection group. A corollary of the main result is a characterization of
the surjectivity of $\rho_X$ in terms of the smoothness of the variety of
$G$-orbits in the $G$-saturation of $X$ in $V$.

\autoref{thm:main} in the present paper was stated and proved for finite
Coxeter groups in~\cite{douglassroehrle:invariants}. The arguments in this
paper are almost entirely independent of that paper and give a more uniform
proof of \cite[Theorem 2.3] {douglassroehrle:invariants}.

Lehrer and Springer \cite{lehrerspringer:intersection} show that if $G$ acts
on $V$ as a reflection group and $X$ is a subspace of $V$ that is maximal
among the eigenspaces of elements of $G$ with a fixed eigenvalue, then the
conclusions of \autoref{thm:main} hold. Because we consider fixed point
subspaces, that is, $1$-eigenspaces, of elements of $G$, the only subspace
covered by both our arguments and those in
\cite{lehrerspringer:intersection} is $V$ itself.

The rest of this paper is organized as follows. In \autoref{sec:state} we
set notation and state the main theorems. We make some reductions and prove
some preliminary results in \autoref{sec:prelim}, then we complete the proof
of \autoref{thm:main} for the infinite family $G(r,p,n)$ in
\autoref{sec:grpn} and for the exceptional unitary reflection groups in
\autoref{sec:exc}. Tables containing the results of our computations for the
exceptional groups that are used in the proof of the theorem are given in an
appendix.

\section{Statement of the main results} \label{sec:state}

Suppose $V$ is an $n$-dimensional complex vector space and $G$ is a finite
subgroup of the general linear group $\GL(V)$. A linear transformation $r\in
\GL(V)$ is a reflection if its fixed point subspace, $\Fix(r)$, is a
hyperplane in $V$. Let $G^\r$ be the subgroup of $G$ generated by the set of
reflections in $G$.  If $G= G^\r$, then we say that $G$ acts on $V$ as a
unitary reflection group, or simply as a reflection group. The group
$\GL(V)$ acts on the algebra $\BBC[V]$ of polynomial functions on $V$ with
\[
(g\cdot f)(v)= f(g\inverse v)\quad\text{for $g\in \GL(V)$,
$f\in \BBC[V]$, and $v\in V$,}
\]
and one may consider the subalgebra $\BBC[V]^G$ of $G$-invariant polynomial
functions on $V$. Obviously a polynomial function $f$ is $G$-invariant if
and only if it is constant on $G$-orbits in $V$. It follows from the
well-known theorem of Chevalley-Shephard-Todd that the following statements
are equivalent:
\begin{itemize}
\item $G$ acts on $V$ as a unitary reflection group.
\item $\BBC[V]^G$ is a polynomial algebra.
\item The orbit variety $V/G$ is smooth.
\end{itemize}
When these conditions hold, it is known that the multiset of degrees of a
set of homogeneous generators of $\BBC[V]^G$ does not depend on the chosen
set of generators. These positive integers are the degrees of $G$. If $d_1,
\dots, d_n$ are the degrees of $G$, then $d_1-1, \dots, d_n-1$ are the
exponents of $G$. In general, define the exponents of a finite subgroup $G$
of $\GL(V)$ to be the multiset of exponents of the unitary reflection group
$G^\r$ and denote this multiset by $\exp(G)$.

By a hyperplane arrangement in $V$ we mean a finite set of hyperplanes in
$V$. The reflections in $G$ determine a hyperplane arrangement in $V$,
namely
\[
\CA(V,G)= \{\, \Fix(r)\mid \text{$r$ is a reflection in $G$} \,\}.
\]
The arrangement $\CA(V,G)$ is called a reflection arrangement. Note that
$G^\r$ acts on $V$ as a reflection group, $\CA(V, G) = \CA(V, G^\r)$ is a
reflection arrangement, and $\exp(G)= \exp(G^\r)$.

For a subspace $X$ of $V$ there are two natural hyperplane arrangements
arising from the action of $G$ on $V$. These are
\begin{enumerate}
\item the reflection arrangement $\CA(X, C_X) =\CA(X, C_X^\r)$ and
\item the restricted arrangement $\CA(V,G)^X$, which consists of the
  intersections $H\cap X$ for $H$ in $\CA(V,G)$ with $X\not\subseteq H$.
\end{enumerate}
By definition $\CA(X, C_X)$ is a reflection arrangement, but in general
$\CA(V,G)^X$ is not necessarily a reflection arrangement.

We can now state our main result.

\begin{theorem}\label{thm:main}
  Suppose $G$ acts on $V$ as a unitary reflection group and $X$ is the space
  of fixed points of an element of $G$.  Then the restriction map
  \[
  \rho_X\colon \BBC[V]^G \to \BBC[X]^{C_X}\text{\ is surjective}
  \]
  if and only if
  \[
  \CA(X,C_X) =\CA(V,G)^X\quad \text{and} \quad \exp(C_X) \subseteq \exp(G).
  \]
  Furthermore, when these conditions hold, $C_X=C_X^\r$ acts on $X$ as a
  reflection group.
\end{theorem}

Examples show that the theorem is sharp in the sense that there are spaces
of fixed points such that $\CA(X,C_X) =\CA(V,G)^X$ and $\exp(C_X)\not
\subseteq \exp(G)$; and there are spaces of fixed points such that
$\CA(X,C_X) \ne\CA(V,G)^X$ and $\exp(C_X) \subseteq \exp(G)$.

The map $\rho_X$ is the comorphism of the finite morphism from the quotient
variety $X/C_X$ to the quotient $V/G$ that maps a $C_X$-orbit in $X$ to its
$G$-orbit in $V$. The image of this morphism is the variety $GX/G$, the
orbit variety of the $G$-saturation of $X$ in $V$, and it factors as the
composition
\[
X/C_X \to GX/G \to V/G,
\]
where the first morphism is surjective and the second is injective.  The
surjective morphism $X/C_X \to GX/G$ is the normalization of the affine
variety $GX/G$. Richardson~\cite[2.2.1]{richardson:normality} has shown that
if $G$ is a unitary reflection group and $X$ is any subspace of $V$, then
$\rho_X$ is surjective if and only if $GX/G$ is a normal variety, or
equivalently, if and only if $X/C_X \cong GX/G$. It follows from the
Chevalley-Shephard-Todd Theorem that $C_X$ acts on $X$ as a reflection group
if and only if $X/C_X$ is a smooth variety. Together with
\autoref{thm:main}, this proves the following corollary.

\begin{corollary}\label{cor:main1}
  Suppose $G$ acts on $V$ as a unitary reflection group and $X$ is the space
  of fixed points of an element of $G$.  Then the following are equivalent:
  \begin{enumerate}
  \item $\rho_X\colon \BBC[V]^G \to \BBC[X]^{C_X}$ is surjective.
  \item $\CA(X,C_X)= \CA(V,G)^X $ and $\exp(C_X) \subseteq \exp(G)$.
  \item $X/C_X \cong GX/G$.
  \item $GX/G$ is a normal variety.
  \item $GX/G$ is a smooth variety.
  \end{enumerate}
\end{corollary}

The lattice of a hyperplane arrangement $\CA$, denoted by $L(\CA)$, is the
set of subspaces of the form $H_1\cap \dotsm \cap H_m$, where $\{ H_1,
\dots, H_m \}$ is a subset of $\CA$. The next lemma is well-known.

\begin{lemma}
  Suppose that $G$ acts on $V$ as a unitary reflection group and that $X$ is
  a subspace of $V$. Then the following statements are equivalent:
  \begin{enumerate}
  \item $X$ is in the lattice of $\CA(V,G)$.
  \item $X$ is the fixed point space of an element of $G$.
  \item $X$ is the fixed point space of a subgroup of $G$.
  \end{enumerate}
\end{lemma}

\begin{proof}
  It is shown in \cite[Theorem~ 6.27]{orlikterao:arrangements} that the
  first two statements are equivalent. The second statement implies the
  third because $\Fix(g) = \Fix( \langle g\rangle)$ for $g\in G$. Finally,
  the third statement implies the first because if $H$ is a subgroup of $G$,
  then $\Fix(H)= \bigcap_{h\in H} \Fix(h)$, and by the second statement each
  subspace $\Fix(h)$ is in $L(\CA(V,G))$.
\end{proof}

Using the lemma, in the following, we frequently refer to subspaces in the
lattice of $\CA(V,G)$ instead of fixed point subspaces of elements, or
subgroups, of $G$. To simplify the notation, in the rest of this paper set
\[
\CA(G)= \CA(V,G), \quad\text{and} \quad \CA(C_X) =\CA(X, C_X)
\]
when $G$ and $V$ are fixed.

In ~\autoref{sec:prelim} it is shown that under the assumptions of the
theorem, if $\rho_X$ is surjective, then $\CA(C_X)= \CA(G)^X$. This focuses
attention on subspaces $X$ such that $\CA(C_X)= \CA(G)^X$. The subspaces in
the lattice of $\CA(G)$ with the property that $\CA(C_X)= \CA(G)^X$ are
classified for the infinite family of irreducible unitary reflection groups
in \autoref{sec:grpn} and for the thirty-four exceptional unitary reflection
groups in \autoref{sec:exc}. One consequence of the classification
is~\autoref{thm:ref}, which states that if either $\CA(C_X)= \CA(G)^X$ or
$\exp(C_X) \subseteq \exp(G)$, then $C_X$ acts on $X$ as a reflection
group. A second consequence of the computations in \autoref{sec:grpn} is an
elementary arithmetic characterization of the subspaces $X\in L(\CA(G))$
such that $C_X$ acts on $X$ as a reflection group when $G=G(r,p,n)$ is in
the infinite family of reflection groups (see \autoref{cor:cxref}). For the
exceptional reflection groups, the subspaces $X$ such that $C_X$ acts on $X$
as a reflection group are listed in an appendix.

\section{Reductions and preliminary results} \label{sec:prelim}

For general information about hyperplane arrangements and reflection groups
we refer the reader to \cite{orlikterao:arrangements}.

As above, suppose $V$ is a finite dimensional complex vector space, $G$ is a
finite subgroup of $\GL(V)$, and $X$ is a subspace of $V$. We assume also
that a positive, definite, hermitian form on $V$ is given and that $G$ is a
subgroup of the unitary group of $V$ with respect to this form.

In this section we make several reductions; state and prove
\autoref{thm:surj}, a preliminary result that may be viewed as a
strengthening of~\cite[Proposition 3.1] {douglassroehrle:invariants}; state
a second preliminary result, \autoref{thm:ref}; and complete the proof
of the forward implication in~\autoref{thm:main}, assuming the validity of
\autoref{thm:ref}. The proof of \autoref{thm:ref} and the completion of the
proof of~\autoref{thm:main} are given in~\autoref{sec:grpn} for the infinite
family of irreducible unitary reflection groups and in~\autoref{sec:exc} for
the exceptional irreducible unitary reflection groups.

\subsection*{Reductions}\label{ssec:red}

In this subsection we show that it is enough to prove~\autoref{thm:main}
when $X$ is chosen from a set of orbit representatives for the action of the
normalizer of $G$ in $\GL(V)$ on $L(\CA(G))$ and that it is enough to
prove~\autoref{thm:main} when $G$ acts faithfully on $V$ as an irreducible
reflection group.

Suppose $h\in \GL(V)$ normalizes $G$. Then $hZ_X h\inverse = Z_{h(X)}$ and
$hN_X h\inverse =N_{h(X)}$, so conjugation by $h$ induces an isomorphism
\[
c_h\colon C_{X} \xrightarrow{\ \cong\ } C_{h(X)}.
\]
Also, the linear transformation $h$ determines an algebra automorphism
\[
h^\#\colon \BBC[V] \xrightarrow{\ \cong\ } \BBC[V]\quad \text{by} \quad
h^\#(f)=f\circ h.
\]
The proof of the next proposition is straightforward and is omitted. 

\begin{proposition}\label{pro:redorb}
  Suppose $G\subseteq \GL(V)$ acts on $V$ as a finite reflection group,
  $h\in \GL(V)$ normalizes $G$, and $X \in L(\CA(G))$. Then
  \begin{enumerate}
  \item $c_h\colon C_{X} \to C_{h(X)}$ is an isomorphism that restricts to
    an isomorphism of reflection groups $C_X^\r\cong C_{h(X)}^\r$,
  \item $c_h$ determines bijections between $\CA(C_X)$ and $\CA(C_{h(X)})$,
    and between $L(\CA(C_X))$ and $L(\CA(C_{h(X)}))$,
  \item $h|_X\colon X\to h(X)$ determines a bijection between $\CA(G)^X$ and
    $\CA(G)^{h(X)}$, and
  \item $\rho_X$ is surjective if and only if $\rho_{h(X)}$ is surjective.
  \end{enumerate}
  Consequently, the conclusions of~\autoref{thm:main} hold for $X$ if and
  only if they hold for $h(X)$.
\end{proposition}

In the rest of this subsection we assume that $G$ is a finite unitary
reflection group and $X$ is in the lattice of $\CA(G)$.

Set $V_f=\Fix(G)$ and let $V_r$ be the orthogonal complement of $V_f$ in
$V$. Obviously $V_f\subseteq X$. Let $X_r$ be the orthogonal complement of
$V_f$ in $X$. The restriction maps $\rhoXtilde$ and $\rho_X$ may be
identified with
\[
\widetilde \rho_{X_r}\otimes \id \colon \BBC[V_r]\otimes \BBC[V_f] \to
\BBC[X_r]\otimes \BBC[V_f]
\]
and
\[
\rho_{X_r}\otimes \id \colon
\BBC[V_r]^G\otimes \BBC[V_f] \to \BBC[X_r]^{C_X}\otimes \BBC[V_f],
\]
respectively, where $\widetilde \rho_{X_r}$ and $\rho_{X_r}$ are given by
restriction.  Clearly, $\rho_X$ is surjective if and only if $\rho_{X_r}$ is
surjective.

Because $V_f$ is contained in every reflecting hyperplane of $G$ in $V$,
there are canonical bijections between $\CA(V,G)$ and $\CA(V_r,G)$, and
between $\CA(V,G)^X$ and $\CA(V_r,G)^{X_r}$. Let $O_f$ denote the multiset
containing $0$ with multiplicity $\dim V_f$. Then
\[
\exp(G,V)= \exp(G,V_r)\amalg O_f\quad \text{and} \quad \exp(C_X, X)=
\exp(C_X, X_r)\amalg O_f.
\]
Thus, $\CA(C_X) =\CA(G)^X$ and $\exp(C_X) \subseteq \exp(G)$ if and only if
$ \CA(X_r,C_X) =\CA(V_r, G)^{X_r}$ and $\exp(C_X, X_r) \subseteq
\exp(G,V_r)$.

It follows from the observations in the preceding two paragraphs that it is
enough to prove~\autoref{thm:main} when $G$ acts faithfully on $V$.

Finally, suppose $G$ acts faithfully on $V$ as a reducible reflection
group. Then there are subgroups $G_1$ and $G_2$ of $G$, complementary
orthogonal subspaces $V_1$ and $V_2$ of $V$, and complementary orthogonal
subspaces $X_1$ and $X_2$ of $X$ such that (1) $G_i$ acts faithfully on $V_i$ as
a unitary reflection group for $i=1,2$, (2) $G\cong G_1\times G_2$, the action
of $G$ on $V$ may be identified with the action of $G_1\times G_2$ on
$V_1\oplus V_2$, and $X\cong X_1\oplus X_2$ of $V_1\oplus V_2$.  It is
straightforward to check that~\autoref{thm:main} holds for $V$, $G$, and $X$
if and only if it holds for $V_i$, $G_i$, and $X_i$ for $i=1,2$. Thus, it is
enough to prove~\autoref{thm:main} when $G$ acts on $V$ as an irreducible
reflection group.

Recall that the irreducible unitary reflection groups are classified in
~\cite{shephardtodd:finite} as one infinite, three parameter family, 
the groups $G(r,p,n)$, and 34 exceptional groups denoted by $G_4$,
\dots, $G_{37}$.

\subsection*{Proof that if \texorpdfstring{$\mathbf{\rho_X}$} {rho} is
  surjective, then \texorpdfstring{$\mathbf{\CA(C_X)= \CA(G)^X}$} {ACX=AX}}

In~\cite[\S3] {douglassroehrle:invariants} it is shown that if $G$ and $C_X$
act as unitary reflection groups, $\CA(C_X) \subseteq \CA(G)^X$, and
$\rho_X$ is surjective, then $\CA(G)^X = \CA(C_X)$ and $\exp(C_X) \subseteq
\exp(G)$. In this subsection we prove a variant of this result that does not
include the hypothesis that $C_X$ acts on $X$ as a reflection group or the
conclusion that $\exp(C_X) \subseteq \exp(G)$.

\begin{theorem}\label{thm:surj}
  Suppose $G$ acts on $V$ as a unitary reflection group, $X$ is in the
  lattice of $\CA(G)$, and the restriction map $\rho_X\colon \BBC[V]^G \to
  \BBC[X]^{C_X}$ is surjective. Then $\CA(C_X) = \CA(G)^X$.
\end{theorem}

Before proving the theorem we need some preliminary results.

First, suppose $Y$ and $Z$ are complex, affine varieties and $\varphi\colon
Y\to Z$ is a finite, surjective morphism. Let $\varphi^\#\colon \BBC[Z]\to
\BBC[Y]$ be the comorphism. Then $\varphi^\#$ is injective and we may
consider $\BBC[Z]$ as a subalgebra of $\BBC[Y]$. Suppose $\fp$ is a prime
ideal in $\BBC[Y]$ and $\fq= \fp \cap \BBC[Z]$ is the contraction of $\fp$
to $\BBC[Z]$. Following \cite{auslanderbuchsbaum:ramification}, we say that
$\fp$ in $\BBC[Y]$ is unramified over $\BBC[Z]$ if firstly $\fq$ generates
the maximal ideal $\fp \BBC[Y]_{\fp}$ of $\BBC[Y]_{\fp}$, where
$\BBC[Y]_{\fp}$ is the localization of $\BBC[Y]$ at $\fp$, and secondly
$\BBC[Y]_{\fp} / \fp \BBC[Y]_{\fp}$ is a separable field extension of
$\BBC[Z]_{\fq} / \fq \BBC[Z]_{\fq}$. If $\fp$ is not unramified, then it is
ramified. Let $y\in Y$ and let $\fm_{Y,y}$ be the maximal ideal of functions
in $\BBC[Y]$ that vanish at $y$. The morphism $\varphi$ is said to be
ramified or unramified at $y$ if $\fm_{Y,y}$ is ramified or unramified over
$\BBC[Z]$.

\begin{lemma}\label{lem:ram}
  Suppose $\varphi\colon Y\to Z$ is a finite, surjective morphism of affine
  varieties.
  \begin{enumerate}
  \item Let $y\in Y$. Then $\varphi$ is unramified at $y$ if and only if the
    map on Zariski tangent spaces $d\varphi_y\colon T_yY \to T_{\varphi(y)}
    Z$ is injective. \label{lem:ram:1}
  \item Let $\fp$ be a prime ideal in $\BBC[Y]$ and let $\CZ(\fp)$ denote
    the zero set of $\fp$ in $Y$. Then $\fp$ is ramified over $\BBC[Z]$ if
    and only if $\varphi$ is ramified at $y$ for every $y\in
    \CZ(\fp)$. \label{lem:ram:2}
  \end{enumerate}
\end{lemma}

\begin{proof}
  It follows from \cite[Proposition 3.6 (i)]{altmankleiman:introduction}
  that $\varphi$ is unramified at $y$ if and only if the cotangent map $
  T_{\varphi(y)}^* Z \to T_y^* Y$ is surjective. This implies the first
  statement.

  Let $\FH= \FH_{\BBC[Y]/\BBC[Z]}$ denote the homological different defined
  by Auslander and Buchsbaum \cite[\S2]{auslanderbuchsbaum:ramification}. It
  is shown in \cite[Theorem 2.7]{auslanderbuchsbaum:ramification} that a
  prime ideal $\fr\subseteq \BBC[Y]$ is ramified over $\BBC[Z]$ if and only
  if $\FH \subseteq \fr$. The second statement now follows from the fact
  that $\fp = \cap_{y\in \CZ(\fp)} \fm_{Y,y}$.
\end{proof}

Returning to the setup in the statement of the theorem, suppose that $\dim
X=a$, choose $g\in G$ such that $X=\Fix(g)$, and choose a basis $\{b_1,
\dots, b_n\}$ of $V$ consisting of eigenvectors for $g$, indexed so that
$\{b_1, \dots, b_a\}$ is a basis of $X$. Let $\{x_1, \dots, x_n\}$ denote
the dual basis of $V^*$. Then the restrictions of $x_1$, \dots, $x_a$ to $X$
form a basis of $X^*$.

\begin{lemma}\label{lem:pz}
  Suppose $a+1\leq j\leq n$. Then $\frac {\partial f} {\partial x_j}|_X = 0$
  for all $f\in \BBC[V]^G$.
\end{lemma}

\begin{proof}
  Note that $X=\Fix(g)$ is the $1$-eigenspace of $g$. Denef and Loeser
  \cite{denefloeser:regular} have shown that if $h\in G$, $v_1$ and $v_2$
  are eigenvectors for $h$ with eigenvalues $\lambda_1$ and $\lambda_2$
  respectively, and $f\in \BBC[V]^G$ is homogeneous with degree $d$, then
  $\lambda_2 D_{v_2}(f)(v_1)= \lambda_1^{1-d} D_{v_2}(f)(v_1)$, where
  $D_v(f)$ denotes the directional derivative of $f$ in the direction of
  $v$. Taking $h=g$, $v_2= b_j$, and $v_1= x\in X$, we have $\lambda_2
  D_{b_j}(f)(x)= D_{b_j}(f)(x)$ where $\lambda_2 \ne 1$. It follows that $0=
  D_{b_j}(f)(x)= \frac {\partial f} {\partial x_j}|_x$ for all $x\in X$ and
  all $f\in \BBC[V]^G$.  
\end{proof}

\begin{proof}[Proof of~\autoref{thm:surj}]
  As shown in \cite{douglassroehrle:invariants}, it is always the case that
  $\CA(C_X) \subseteq \CA(G)^X$: Suppose $K$ is in $\CA(C_X)$. By assumption
  there is a $g$ in $N_X$ so that $\Fix(g)\cap X= K$. Then $\Fix(g)$ is in
  the lattice of $\CA$, say $\Fix(g)= H_1\cap \dotsm\cap H_k$, where $H_1$,
  \dots, $H_k$ are in $\CA$. Thus $K= H_1\cap \dots\cap H_k \cap X$. Since
  $\dim K=\dim X-1$, it follows that $K=H_i\cap X$ for some $i$ and so $K$
  is in $\CA(G)^X$.

  Conversely, we need to show that $ \CA(G)^X \subseteq \CA(C_X)$. Set
  \[
  X_\ram= \bigcup_{K\in \CA(G)^X} K
  \]
  and choose a set of homogeneous generators $\{f_1 \dots, f_n\}$ of
  $\BBC[V]^G$. Consider the commutative diagram
  \[
  \xymatrix{ X \ar[r]^-{\pi_1} \ar[d]^{i} & X / C_X \ar[r]^{\pi_2}_{\cong} &
    GX / G \ar[d]^{i_1} \ar[r]^-{\overline F}_{\cong} & F(GX/G) \ar[d]^{i_2}
    \\
    V \ar[rr]^{\pi} & & V/G \ar[r]^{F}_{\cong} & \BBC^n }
  \]
  where the maps are defined as follows: $\pi$, $\pi_1$, and $\pi_2$ are the
  natural quotient maps; $i$, $i_1$, and $i_2$ are the inclusions; $F(Gv)_i
  = f_i(v)$ for $1\leq i\leq n$; and $\overline F$ is obtained from $F$ by
  restriction. Note that $\rho_X$ is the comorphism of $i_1 \circ
  \pi_2$. Richardson \cite[2.2.1]{richardson:normality} has shown that
  $\rho_X$ is surjective if and only if $\pi_2$ is an isomorphism and so in
  our situation, $\pi_2$ is an isomorphism. It follows from the theorem of
  Chevalley-Shephard-Todd that $F$ is an isomorphism, hence, so is
  $\overline F$.

  Suppose $x\in X$. Then we have a commutative diagram of Zariski tangent
  spaces
  \[
  \xymatrix{ T_x X \ar[r]^-{d(\pi_1)_x} \ar[d]^{di_x} & T_{\pi_1(x)} X / C_X
    \ar[r] & T_{\pi(x)} GX / G \ar[d] \ar[r] & T_{F(\pi(x))} F(GX/G) \ar[d]
    \\ 
    T_x V \ar[rr]^{d\pi_x} & & T_{\pi(x)} V/G \ar[r]^{dF_{\pi(x)}} &
    T_{F(\pi(x))} \BBC^n .}
  \]
  The matrix of the composition $dF_{\pi(x)}\circ d\pi_x$ with respect to
  suitable bases is given by the $n\times n$ Jacobian matrix $J(F\pi)(x)$,
  whose $(i,j)$-entry is $\frac {\partial f_i} {\partial x_j}|_x$. Steinberg
  \cite{steinberg:invariants} has shown that the nullity of $J(F\pi)(x)$ is
  the maximum number of linearly independent hyperplanes in $\CA(G)$ that
  contain $x$. Thus the rank of $J(F\pi)(x)$ is at most $a$ and is equal to
  $a$ if and only if $x\not \in X_\ram$. It follows from \autoref{lem:pz}
  that the last $n-a$ columns of $J(F\pi)(x)$ are zero and so the matrix of
  the composition $dF_{\pi(x)} \circ d\pi_x \circ di_x$ with respect to
  suitable bases is the $n\times a$ matrix obtained from $J(F\pi)(x)$ by
  deleting the last $n-a$ columns. It follows that the rank of $d(\pi_1)_x$
  is at most $a$ and equal to $a$ if and only if $x\not \in
  X_\ram$. Therefore, $d(\pi_1)_x$ is injective if and only if $x\not \in
  X_\ram$ and so it follows from \autoref{lem:ram}\autoref{lem:ram:1} that
  $\pi_1$ is ramified at $x$ if and only if $x\in X_\ram$.

  Now suppose that $K\in \CA(G)^X$ and let $\CI(K)$ be the ideal of $K$ in
  $\BBC[X]$. Then $\CI(K)$ is a prime ideal in $\BBC[X]$ with height one and
  it follows from \autoref{lem:ram}\autoref{lem:ram:2} that $\CI(K)$ is
  ramified over $\BBC[X/C_X]= \BBC[X]^{C_X}$. Benson
  \cite[\S3.9]{benson:polynomial} has shown that if $\fp\subseteq \BBC[X]$
  is a height one prime ideal, then $\fp$ is ramified over $\BBC[X]^{C_X}$
  if and only if $\fp$ is the ideal of a reflecting hyperplane of
  $C_X$. Thus, $K\in \CA(C_X)$ as claimed.
\end{proof}

\subsection*{The subgroups \texorpdfstring{ $\mathbf{
      C_X^{\operatorname{\mathbf {ref}}}}$} {cxref}}

The subgroups $C_X^\r$ are determined in \autoref{pro:cxref} for the groups
$G(r,p,n)$ and are listed in the tables in the appendix for the exceptional
irreducible unitary reflection groups, and the subspaces $X\in L(\CA(G))$
such that $C_X^\r=C_X$ are characterized in \autoref{cor:cxref} for the
infinite family of irreducible unitary reflection groups and are listed in
the tables in the appendix for the exceptional irreducible unitary
reflection groups.

\begin{theorem}\label{thm:ref}
  Suppose $G$ acts on $V$ as a unitary reflection group and $X$ is in the
  lattice of $\CA(G)$. If either $\CA(C_X)= \CA(G)^X$ or $\exp(C_X)\subseteq
  \exp(G)$, then $C_X^\r=C_X$ and thus $C_X$ acts on $X$ as a unitary
  reflection group.
\end{theorem}

\begin{proof} 
  As in the previous subsection it is enough to prove the theorem when $G$
  acts faithfully on $V$ as an irreducible reflection group. If $G=G(r,p,n)$
  for some $r,p,n$, then the result follows from \autoref{cor:exp} and
  \autoref{cor:ref}. If $G$ is one of the exceptional unitary reflection
  groups, then the result is proved in \autoref{pro:excref}.
\end{proof}

Examples show that the converse of the theorem is false; see
\autoref{rem:table}.

\subsection*{Proof of the forward implication in~\autoref{thm:main}}

Fix $X$ in $L(\CA(G))$ and suppose that $\rho_X\colon \BBC[V]^{G} \to
\BBC[X]^{ C_X}$ is surjective. We may assume that $G$ acts faithfully on
$V$. By~\autoref{thm:surj}, $\CA(C_X) = \CA(G)^X$ and so
by~\autoref{thm:ref}, $C_X$ acts on $X$ as a reflection group. Thus
$\BBC[V]^{G}$ and $\BBC[X]^{ C_X}$ are polynomial algebras and
by~\cite[Lemma 4.1]{richardson:normality}, if $\dim X=a$, there are
homogeneous generators $\{f_1, \dots, f_n\}$ of $\BBC[V]^{G}$ so that
$\{\rho_X(f_1), \dots, \rho_X(f_a)\}$ is a set of homogeneous generators of
$\BBC[X]^{C_X}$. It follows that $\exp(C_X) \subseteq \exp(G)$. Assuming the
validity of \autoref{thm:ref}, this completes the proof that if $\rho_X$ is
surjective, then $\CA(C_X) = \CA(G)^X$ and $\exp(C_X) \subseteq \exp(G)$.

\section{The infinite family of irreducible unitary reflection
  groups} \label{sec:grpn}

In this section we suppose that $G$ is in the infinite family of
irreducible unitary reflection groups and we classify the subspaces $X\in
L(\CA(G))$ such that $\CA(C_X) = \CA(G)^X$, we classify the subspaces $X\in
L(\CA(G))$ such that $C_X$ acts on $X$ as a reflection group, and we
complete the proof of~\autoref{thm:main} for these groups.

For a positive integer $k$, let $\mu_k$ denote the group of $k\th$ roots of
unity in $\BBC$ and set $[k]=\{1,2, \dots, k\}$. Throughout this section $r$
is a positive integer, $p$ is a positive divisor of $r$, and $q=r/p$. 

The presentation in the next two subsections is a reformulation of the
constructions and results in ~\cite[\S6.4]{orlikterao:arrangements} and
\cite[\S3]{taylor:reflection}. Proofs of the assertions are straightforward
and are omitted.

\subsection*{The groups \texorpdfstring{$\mathbf{ G (r,p,n)}$}{Grpn} and the
  arrangements \texorpdfstring{$\mathbf{ \CA_{n}^k}$} {CAn}}

Let $V = \BBC^n$. Using the standard basis $\{e_1, \dots, e_n\}$ of $\BBC^n$
we identify $\GL(V)$ with the matrix group $\GL_n(\BBC)$. For $1\leq i\leq
n$, let $x_i\colon \BBC^n\to \BBC$ be projection on the $i\th$
coordinate. For $\sigma$ in the symmetric group $S_n$, let $p_\sigma$ be the
permutation matrix defined by $p_\sigma e_i=e_{\sigma(i)}$ for $1\leq i\leq
n$. Let $W$ denote the set of permutation matrices in $\GL_n(\BBC)$ and let
$D$ denote the group of diagonal matrices in $\GL(V)$ with entries in
$\mu_r$. For $t\in D$ and $1\leq i\leq n$, let $t_i$ denote the $(i,i)$
entry of $t$. Similarly, for $v\in V$, $v_i$ denotes the $i\th$ coordinate
of $v$.

Define 
\[
\omega= e^{2\pi \sqrt {-1}/r}.
\]
Then $\omega$ is a generator of $\mu_r$ and $\omega^p$ is a generator of
$\mu_q$. Note that the determinant map, $\det\colon D\to \mu_r$, is a group
homomorphism. Define $D_p$ to be the preimage of the subgroup $\mu_q$. Then
\[ 
D_p = \{\,t \in D \mid \det t \in \mu_{q}\,\}
\]
and $W$ normalizes $D_p$. Set
\[
G(r,p,n)=WD_p.
\]
Note that every element in $G(r,p,n)$ has a unique factorization of the form
$p_\sigma t$ where $\sigma\in S_n$ and $t\in D_p$. As $n$ and $r$ are fixed,
we denote $G(r,p,n)$ simply by $G_p$. Obviously,
\[
G_r \subseteq G_p \subseteq G_1.
\]

Suppose $1\leq i\ne j\leq n$ and $\zeta$ is in $\mu_r$. Define
$r_{ij}(\zeta)$ to be the $n\times n$ matrix whose $(k,l)$-entry is
\[
r_{ij}(\zeta)_{kl}=
\begin{cases}
  1&k=l,\ k\ne i,\ k\ne j\\
  \zeta&(k,l)=(i,j)\\
  \zeta\inverse &(k,l)=(j,i)\\
  0&\text{otherwise.}
\end{cases}
\]
It is easy to check that the characteristic polynomial of $r_{ij}(\zeta)$ is
$(x-1)^{n-1} (x+1)$. It follows that $r_{ij}(\zeta)$ is a reflection with
order two. Define
\[
H_{ij}(\zeta)=\ker (x_i-\zeta x_j)= \{\, v\in V\mid v_i= \zeta v_j\,\}.
\]
It is easy to see that $H_{ij}(\zeta)$ is the reflecting hyperplane of
$r_{ij}(\zeta)$.

Suppose $1\leq i\leq n$ and $\zeta$ is in $\mu_r$ with $\zeta \ne 1$. Define
$r_{ii}(\zeta)$ to be the diagonal $n\times n$ matrix whose $(k,k)$-entry is
\[
r_{ii}(\zeta)_{kk}=
\begin{cases}
  \zeta&k=i\\
  1&\text{otherwise.}
\end{cases}
\]
Then $r_{ii}(\zeta)$ is obviously a reflection. Define
\[
H_{i}=\ker x_i= \{\, v\in V \mid v_i=0\,\}.
\]
It is clear that $H_{i}$ is the reflecting hyperplane of $r_{ii}(\zeta)$ for
all $\zeta$ in $\mu_r$ with $\zeta\ne1$.

Obviously, $r_{ij}(\zeta)$ is in $G_p$ for all $\zeta\in \mu_r$ and
$r_{ii}(\zeta)$ is in $G_p$ if and only if $\zeta\in \mu_q$. It can be shown
that these are all the reflections in $G_p$. Note that $G_r$ contains no
reflections of the form $r_{ii}(\zeta)$. When $p<r$, the reflections in
$G_p$ and their reflecting hyperplanes may be arranged in the arrays
\begin{equation}
  \label{eq:refp}  \scriptsize
  \begin{matrix}
    r_{11}(\zeta_1)& r_{12}(\zeta_{12}) & r_{13}(\zeta_{13}) & \dots &
    r_{1n}(\zeta_{1n})\\
    & r_{22}(\zeta_{2}) & r_{23}(\zeta_{23}) & \dots & r_{2n}(\zeta_{2n})\\
    && r_{33}(\zeta_{3}) & \dots & r_{3n}(\zeta_{3n})\\
    &&&\ddots&\vdots\\
    &&&& r_{nn}(\zeta_{n})
  \end{matrix} {\text{\normalsize\ \  and} }
  \begin{matrix}
    H_1& H_{12}(\zeta_{12}) & H_{13}(\zeta_{13}) & \dots &
    H_{1n}(\zeta_{1n})\\
    & H_{2} & H_{23}(\zeta_{23}) & \dots & H_{2n}(\zeta_{2n})\\
    && H_{3} & \dots & H_{3n}(\zeta_{3n})\\
    &&&\ddots&\vdots\\
    &&&& H_{n}
  \end{matrix}
\end{equation}
where $\zeta_{ij}\in \mu_{r}$ for $1\leq i<j\leq n$ and $\zeta_i\in \mu_q
\setminus \{1\}$ for $1\leq i\leq n$. Similarly, when $p=r$, the reflections
in $G_r$ and their reflecting hyperplanes are
\[
\{\, r_{ij}(\zeta_{ij}) \mid 1\leq i<j\leq n,\ \zeta_{ij}\in \mu_r\,\}
\quad \text{and} \quad
\{\, H_{ij}(\zeta_{ij}) \mid 1\leq i<j\leq n,\ \zeta_{ij}\in \mu_r\,\},
\]
respectively.

Define
\[
\CA_{n}^0= \{\, H_{ij}(\zeta)\mid 1\leq i<j\leq n,\ \zeta\in \mu_r\,\}
\]
and for $1\leq k\leq n$ define
\[
\CA_{n}^k= \{\, H_i\mid 1\leq i\leq k\,\} \amalg \CA_{n}^0.
\]
It is proved in~\cite[\S6.4]{orlikterao:arrangements} that
\[
\CA(G_p) = \CA_{n}^n \quad \text{for $p<r$, and} \quad \CA(G_r)= \CA_{n}^0.
\]

\subsection*{The lattice of the arrangement of \texorpdfstring {$\mathbf{
      G_p}$} {G(r,p,n)} and orbit representatives} \label{subsec:latticerpn}

To simplify formulas later, we extend the definition of $H_{ij}(\zeta)$ to
include all pairs $(i,j)$ with $1\leq i, j\leq n$ by defining
\[
H_{ii}(\zeta)= V \quad\text{for $1\leq i\leq n$.}
\]

Consider a pair $L=(L_0, \{L_1, \dots, L_a\})$ where $\{L_0, L_1, \dots,
L_a\}$ is a collection of disjoint subsets of $[n]$ such that $L_i\ne
\emptyset$ for $1\leq i\leq a$ and $\bigcup_{i=0}^a L_i=[n]$. Notice that
$L_0$ is allowed to be empty and that $\{L_1, \dots, L_a\}$ is a partition
of $[n]\setminus L_0$. Define
\[
X_{L}= \Big(\bigcap_{i\in L_0} H_i\Big) \cap \Big( \bigcap_{j=1}^a
\bigcap_{k,l\in L_j} H_{k,l}(1) \Big).
\]
If $v\in \BBC^n$, then $v$ is in $X_{L}$ if and only if
\begin{enumerate}
\item $v_i=0$ for all $i\in L_0$ and
\item $ v_k= v_l$ for all $k,l$ such that $k,l\in N_j$ for some $1\leq j\leq
  a$.
\end{enumerate}
For $1\leq i\leq a$, define $b_i=\sum_{k\in L_i} e_k$ and let $\CB_L=\{\,
b_i\mid 1\leq i\leq a\,\}$. Then $\CB_L$ is a basis of $X_L$.

Suppose $t\in D$. If $v$ is in $X_L$, then $(tv)_i =t_iv_i=0$ for $i\in L_0$
and if $k,l\in L_j$ for some $j \ge 1$, then $t_k\inverse (tv)_k = v_k =v_l
=t_l\inverse (tv)_l$. Conversely, if $v\in V$ is such that $v_i=0$ for all
$i\in L_0$ and $t_k\inverse v_k=t_l\inverse v_l$ for all $k,l$ such that
$k,l\in L_j$ for some $1\leq j\leq a$, then $v\in t(X_L)$. Thus,
\[
t(X_L)= \Big(\bigcap_{i\in L_0} H_i\Big) \cap \Big( \bigcap_{j=1}^a
\bigcap_{k,l\in L_j} H_{k,l}(t_k/ t_l) \Big).
\]

Let $\CL$ denote the set of all pairs $L=(L_0, \{L_1, \dots, L_a\})$ as
above. The next theorem is a restatement of results of Orlik, Solomon, and
Terao~\cite[\S6.4]{orlikterao:arrangements}.

\begin{theorem}\label{thm:lat}
  The lattice of the arrangement $\CA(G_p)$ is as follows. If $p<r$, then
  \[
  L(\CA(G_p)) =L(\CA_{n}^n)= \{\, t(X_{L}) \mid L\in \CL, \ t\in D\,\}.
  \]
  For $p=r$,
  \[
  L(\CA(G_r))= L(\CA_{n}^0)= \{\, t(X_{L}) \mid L \in \CL,\ |L_0|\ne 1,\
  t\in D\,\}.
  \]
\end{theorem}

The group $G_p$ acts on $L(\CA_{n}^n)$ and $L(\CA_{n}^0)$ is a $G_p$-stable
subset. In the rest of this subsection we describe a subset of $L(\CA_{n}^n)$
that contains at least one representative from each $G_p$-orbit. The main
result is \autoref{cor:or}. Taylor \cite[Theorems 3.9 and
3.11]{taylor:reflection} also describes a set of orbit representatives.

First, the subspaces $t(X_L)$ in \autoref{thm:lat} are not all distinct. For
$L\in \CL$ define $D_L$ to be the setwise stabilizer of $X_L$ in $D$:
\[
D_L=\{\, t\in D\mid t(X_L)=X_L\,\}.
\]
Obviously, $t(X_L)=X_L$ if and only if $tb_i\in X_L$ for $1\leq i\leq a$ and
so 
\[
D_L=\{\, t\in D\mid \forall\, 1\leq j\leq a,\ \forall\, k,l\in L_j,\
t_k=t_l\,\}.
\]
If $L'$ is another pair, then it is easy to see that $X_L=X_{L'}$ if and
only if $L=L'$ and that for $t,t'\in D$,
\[
t(X_{L})= t'(X_{L'})\quad\text{if and only if} \quad L=L'\text{\ and\ }
tD_L= t'D_{L}.
\]

\begin{lemma}\label{lem:conj}
  Suppose $L=(L_0, \{L_1, \dots, L_a\})$ and $\sigma\in S_n$. Then
  \[
  p_\sigma( X_L)= X_{\sigma(L)},
  \]
  where $\sigma(L)= (\sigma(L_0), \{\sigma(L_1), \dots, \sigma(L_a)\} )$.
\end{lemma}

Now suppose $k$ is an integer with $0\leq k\leq n$ and $\lambda=(\ell_1,
\dots, \ell_a)$ is a partition of $k$. Set $\bar \ell_0=0$ and for $i>0$ let
$\bar \ell_i=\ell_1+ \dots + \ell_i$ denote the $i\th$ partial sum of
$\lambda$. Define a pair $L_\lambda = ((L_\lambda)_0, \{(L_\lambda)_1,
\dots, (L_\lambda)_a\})$ in $\CL$ by
\[
(L_\lambda)_0= \{k+1, \dots, n\}\quad \text{and} \quad (L_\lambda)_i= \{
\bar \ell_{i-1}+1, \bar \ell_{i-1}+2, \dots, \bar \ell_{i} \} \text{ for }
1\leq i\leq a.
\]
Also, define
\[
\delta_\lambda= \gcd(\ell_1, \dots, \ell_a) \quad\text{and}\quad
\delta_{\lambda,p} = \gcd(p, \ell_1, \dots, \ell_a),
\]
and set
\[
X_{\lambda}= X_{L_\lambda}, \quad \CB_\lambda=\CB_{L_\lambda}, \quad
\text{and} \quad D_\lambda= D_{L_\lambda}.
\]
Note that $\CB_\lambda=\{\, b_1, \dots, b_a\,\}$ where $b_i= e_{\bar
  \ell_{i-1}+1} + \dots +e_{\bar \ell_i}$.

\begin{lemma}
  If $X\in L(\CA_{n}^n)$, then there is an integer $k$ with $0\leq k\leq n$,
  a partition $\lambda$ of $k$, and a permutation $\sigma\in S_n$ so that
  $p_\sigma(X)$ is in the $D$-orbit of $X_\lambda$.
\end{lemma}

It follows from the lemma that to find representatives of the orbits of the
action of $G_p$ on $L(\CA_{n}^n)$, it is enough to decompose the $D$-orbit
of $X_\lambda$, $DX_\lambda$, into $D_p$-orbits when $\lambda$ is a
partition of $k$ and $0\leq k\leq n$. Now $D_\lambda$ is the stabilizer of
$X_\lambda$ in $D$ and so the $D_p$-orbits in $DX_\lambda$ are in bijection
with $D_p, D_\lambda$-double cosets in $D$. Because $D$ is abelian, these
double cosets are simply the cosets of $D_pD_\lambda$ in $D$.

\begin{proposition} 
  \label{lem:crs}
  Suppose $\lambda=(n_1, \dots, n_a)$ is a partition of $k$ with $0\leq
  k\leq n$.
  \begin{enumerate}
  \item If $k<n$, then $D_pD_\lambda=D$.
  \item If $k=n$, then $|D:D_pD_\lambda| = \delta_{\lambda, p}$. Moreover, if
    for $0\leq u<\delta_{\lambda, p}$, $d_u$ is an element in $D$ with $\det
    d_u =\omega^u$, then $\{\, d_u\mid 0\leq u<\delta_{\lambda, p} \,\}$ is
    a cross-section of $D_pD_\lambda$ in $D$.
  \end{enumerate}
\end{proposition}

We now specify the elements $d_u$ in the preceding lemma, and hence
$G_p$-orbit representatives in $L(\CA_{n}^n)$, as follows. Let $t_0$ be the
diagonal matrix with
\[
(t_0)_i= \begin{cases} \omega & i=1 \\ 1 & i>1 .\end{cases} 
\]
To simplify the notation, set
\[
X_{\lambda,u}= t_0^u(X_\lambda).
\]
Note that $X_{\lambda,0} =X_\lambda$.

\begin{corollary}\label{cor:or}
  Suppose $X\in L(\CA_{n}^n)$. Then there is an integer $k$ with $0\leq
  k\leq n$ and a partition $\lambda$ of $k$ such that $X$ is in the
  $G_1$-orbit of $X_\lambda$.

  If $k<n$, then $X$ is in the $G_p$-orbit of $X_\lambda$.  

  If $k=n$, then there is an integer $u$ with $0\leq u<\delta_{\lambda,p}$
  so that $X$ is in the $G_p$-orbit of $X_{\lambda,u}$.
\end{corollary}

Note that $t_0$ normalizes $G_p$. Thus, the next corollary follows
from~\autoref{pro:redorb} and~\autoref{cor:or}.

\begin{corollary}\label{cor:redor}
  Suppose $X\in L(\CA(G_p))$ and $X$ is in the $G_p$-orbit of
  $X_{\lambda,u}$. Then the conclusion of~\autoref{thm:main} holds for $X$
  if and only if it holds for $X_\lambda$.
\end{corollary}

\subsection*{The groups \texorpdfstring {$\mathbf{ C_{ X}^{\operatorname{
          \mathbf{ref}}}}$ } {Xljref}}

In this subsection we determine the groups $C_X^\r$ as reflection subgroups
of $\GL(X)$, we derive enough information about the groups $C_{X}$ to
characterize the subspaces $X\in L(\CA(G_p))$ such that $C_X$ acts on $X$ as
a reflection group, and we show that if $\exp(C_X)\subseteq \exp(G_p)$, then
$C_X$ acts on $X$ as a reflection group. 

The groups $N_X$, $Z_X$, and $C_X$ depend on the ambient group $G_p$. We
indicate this dependence with a superscript $p$, so
\[
N_X^p= \{\, g\in G_p\mid g(X) =X\,\}, \quad Z_X^p= \{\, g\in G_p\mid g|_X=
\id\,\}, \quad\text{and}\quad C_X^p= N_X^p/ C_X^p.
\]
Note that $N_X^p= G_p \cap N_X^1$ and $Z_X^p= G_p \cap Z_X^1$.

Suppose $0\leq k\leq n$ and $\lambda=(\ell_1, \dots, \ell_a)$ is a partition
of $n-k$.  Let $i_1, \dots, i_c$ be the distinct parts of $\lambda$ with
$i_1> \dots >i_c$ and multiplicities $m_1, \dots, m_c$, respectively, so
that $\lambda= (i_1^{m_1}, \dots, i_c^{m_c})$. Our first task is to describe
the groups $N_{X_\lambda}^p$, $Z_{X_\lambda}^p$, and $C_{X_\lambda}^p$ (see
\autoref{eq:n1}, \autoref{eq:z1}, and \autoref{eq:cxp}). To simplify the
notation, until the statement of \autoref{pro:cxref} set $L=L_\lambda$.

Define
\[
W_\lambda = \{\, p_\sigma\in W\mid \sigma(L_\lambda)= L_\lambda\,\}.
\]
Then $W_\lambda$ is the semidirect product of the subgroups $Z_\lambda$ and
$C_\lambda$, where
\[
Z_\lambda = \{\, p_\sigma\in W_\lambda \mid \forall\ 0\leq i\leq a,\ 
\sigma(L_i)= L_i\,\}
\]
is a normal subgroup of $W_\lambda$ and
\[
C_\lambda =\{\, p_\sigma\in W_\lambda \mid \text{$\sigma(L_0)=L_0$ and
  $\sigma|_{L_i}$ is increasing for $1\leq i\leq a$} \,\}.
\]
In addition, 
\[
Z_\lambda \cong S_{\ell_1} \times \dotsb \times S_{\ell_a} \times S_{n-k}
\quad\text{and} \quad C_\lambda \cong S_{m_1} \times \dotsb \times S_{m_c}.
\]

Clearly $N_{X_\lambda}^1 = W_\lambda D_\lambda$ and so
\begin{equation}
  \label{eq:n1}
  N_{X_\lambda}^p= G_p \cap W_\lambda D_\lambda= W_\lambda (G_p\cap D_\lambda)
  = W_\lambda (D_p\cap D_\lambda) = C_\lambda Z_\lambda (D_p\cap D_\lambda).   
\end{equation}

Next consider $Z_{X_\lambda}^1$. Obviously, $g\in G_1$ is in
$Z_{X_\lambda}^1$ if and only if $gb_i=b_i$ for $b_i\in \CB_\lambda$, and by
a theorem of Steinberg \cite[Theorem 1.5]{steinberg:differential},
$Z_{X_\lambda}^1$ is generated by the reflections it contains. Referring
to~\autoref{eq:refp}, it is easy to see that
\begin{enumerate}
\item for $1\leq l\leq n$ and $\zeta\in \mu_r\setminus \{1\}$,
  $r_{ll}(\zeta)$ is in $Z_{X_\lambda}^1$ if and only if $l\in L_0$,
\item for $1\leq l<l'\leq n$ and $\zeta\in \mu_r$, $r_{ll'}(\zeta)$ is in
  $Z_{X_\lambda}^1$ if and only if either
  \begin{enumerate}
  \item $l,l'\in L_0$, or
  \item there is an $i>0$ so that $l,l'\in L_i$ and $\zeta=1$.
  \end{enumerate}
\end{enumerate}
From this it is straightforward to check that if
\[
D_{\ell_0} =\{\, t\in D\mid \forall\ 1\leq i\leq k,\ t_i=1 \,\},
\]
then $Z_{X_\lambda}^1\cap D=D_{\ell_0}$ and
\[
Z_{X_\lambda}^1 = Z_\lambda D_{\ell_0} \cong S_{\ell_1} \times \dotsb \times
S_{\ell_a} \times G(r,1,n-k).
\]
Therefore
\begin{equation}
  \label{eq:z1}
  Z_{X_\lambda}^p= G_p \cap Z_{X_\lambda}^1 = Z_\lambda (D_p\cap D_{\ell_0}) \cong
  S_{\ell_1} \times \dotsb \times S_{\ell_a} \times G(r,p,n-k).  
\end{equation}

Finally, it follows from~\autoref{eq:n1},~\autoref{eq:z1}, and the
computation of $|D_p\cap D_\lambda|$ in the proof of~\autoref{lem:crs} that
\begin{equation}
  \label{eq:cxp}
  |C_{X_\lambda}^p| =
  \begin{cases}
    m_1!\dotsb m_c! r^a &\text{if  $k<n$}\\
    m_1!\dotsb m_c! \frac {r^a \delta_{\lambda,p}} p & \text{if  $k=n$.}\\
  \end{cases}
\end{equation}

\begin{proposition}\label{pro:cxref}  
  Suppose $X\in L(\CA(G_p))$ is in the $G_p$-orbit of $X_{\lambda,u}$, where
  $\lambda=(\ell_1, \dots, \ell_a)$ is a partition of $k$ for some $0\leq
  k\leq n$ and $u$ is an integer with $0\leq u<\delta_{\lambda,p}$. Suppose
  $i_1, \dots, i_c$ are the distinct parts of $\lambda$ with $i_1> \dots
  >i_c$ and multiplicities $m_1, \dots, m_c$, so $\lambda= (i_1^{m_1},
  \dots, i_c^{m_c})$. Then as a reflection subgroup of $\GL(X)$,
  \[
  C_{X}^\r \cong
  \begin{cases}
    G(r,1,m_1) \times \dotsb \times G(r,1,m_c) &\text{if $k<n$} \\
    G(r,p_{i_1},m_1) \times \dotsb \times G(r,p_{i_c}, m_c) &\text{if
      $k=n$,}
  \end{cases} 
  \]
  where for an integer $j$, $p_{j}=\lcm (j, p)/j$.
\end{proposition}

\begin{proof} 
  By~\autoref{pro:redorb}, without loss of generality we may assume that
  $X=X_\lambda$.

  Suppose $\sigma\in S_n$ and $t\in D$. Then
  \[
  p_\sigma t(X) = (p_\sigma tp_\sigma \inverse) p_\sigma(X_\lambda) =
  (p_\sigma tp_\sigma \inverse) (X_{\sigma(L_\lambda)}).
  \]
  It is shown in \cite[Proposition 6.74]{orlikterao:arrangements} that
  $p_\sigma t(X_\lambda) = X_\lambda$ if and only if $\sigma(L_\lambda)
  =L_\lambda$ and $p_\sigma tp_\sigma \inverse\in D_\lambda$. Now
  $W_\lambda$ normalizes $D_\lambda$ and so
  \begin{equation}
    \label{eq:fa}
    p_\sigma t(X) = X\quad \text{if and only if}\quad p_\sigma\in W_\lambda
    \text{ and } t\in D_\lambda.    
  \end{equation}
  Note that $W_\lambda$ permutes the vectors in the basis
  $\CB_\lambda=\{b_1, \dots, b_a\}$ of $X$ and that each $b_i\in
  \CB_\lambda$ is a common eigenvector for $D_\lambda$.

  For a linear transformation $g\in \GL(X)$ let $[[g]]$ denote the matrix of
  $g$ with respect to $\CB_\lambda$. Then the rule $g\mapsto [[g]]$ defines
  a group homomorphism from $N_{X}^p$ to $\GL_a(\BBC)$ with kernel
  $Z_{X}^p$, and hence an injection from $C_{X}^p$ to $\GL_a(\BBC)$.  It
  follows from~\autoref{eq:fa} that the image of this mapping is contained
  in the subgroup $G(r,1,a)$ of $\GL_a(\BBC)$. In particular, every
  reflection in $C_{X}^p$ is one of the reflections listed
  in~\autoref{eq:refp}.

  Suppose $\zeta\in \mu_r$. If $1\leq i\ne j\leq a$ and $\ell_i=\ell_{j}$, define
  $s_{ij}(\zeta)\in \GL(V)$ by
  \[
  s_{ij}(\zeta) e_l=
  \begin{cases}
    \zeta e_{\bar \ell_{i-1}+ \nu} & \text{if $l= \bar \ell_{j-1} +\nu$,
      $\nu\in [\ell_j]$} \\
    \zeta\inverse e_{\bar \ell_{j-1}+ \nu} & \text{if $l= \bar \ell_{i-1}
      +\nu$,  $\nu\in [\ell_i]$} \\
    e_l & \text{otherwise,}
  \end{cases} 
  \]
  where as above, $\bar \ell_0=0$ and $\bar \ell_i= \ell_1+\dotsb + \ell_i$ for
  $i>0$. Then $s_{ij}(\zeta)$ acts on $\CB_\lambda$ by $b_j\mapsto \zeta
  b_i$, $b_i\mapsto \zeta\inverse b_j$, and $b_l\mapsto b_l$ for $l\ne
  i,j$, and so $s_{ij}(\zeta)\in N_{X}^p$ and $[[s_{ij}(\zeta)]] =
  r_{ij}(\zeta)$. If $1\leq i\ne j\leq n$ and $n_i\ne n_j$, then there is no
  $\sigma\in S_n$ such that $p_\sigma b_i=b_j$ and no element $p_\sigma t\in
  N_{X}^p$ with $[[p_\sigma t]] = r_{ij}(\zeta)$.
 
  Now suppose that $k<n$. For $1\leq i\leq a$ and $\zeta\in \mu_r$ define
  $s_{ii}(\zeta)\in \GL(V)$ by
  \[
  s_{ii}(\zeta) e_l=
  \begin{cases}
    \zeta e_l & \text{if $l\in \{\bar \ell_{i-1}+1 , \dots, \bar \ell_i\}$} \\
    e_l & \text{if $l\in [k] \setminus \{\bar \ell_{i-1}+1 , \dots, \bar
      \ell_i\}$} \\
    \zeta^{-\ell_i} e_{k+1} & \text{if $l=k+1$} \\
    e_l & \text{$k+1<l\leq n$.}
  \end{cases} 
  \]
  Then $s_{ii}(\zeta)$ acts on $\CB_\lambda$ by $b_i\mapsto \zeta b_i$ and
  $b_l\mapsto b_l$ for $l\ne i$, so $s_{ii}(\zeta)\in N_{X}^p$ and
  $[[s_{ii}(\zeta)]] = r_{ii}(\zeta)$ when $\zeta \ne 1$. Because
  $(C_X^p)^\r$ is generated by the reflections it contains, it follows that
  \[
  (C_X^p)^\r\cong G(r,1,m_1) \times \dotsb \times G(r,1,m_c)
  \quad\text{when} \quad k<n.
  \]

  Finally, suppose $k=n$ and $1\leq i\leq a$. If $\zeta\in \mu_r\setminus
  \{1\}$ and $t\in N_X^p$ is such that $[[t]]= r_{ii}(\zeta)$, then $t\in
  D_p$, $tb_i= \zeta b_i$, and $tb_l=b_l$ for $l\ne i$. Thus,
  \[
  \zeta^{n_i}= \det t \in \langle \omega^{n_i} \rangle \cap \langle
  \omega^{p} \rangle = \langle \omega^{\lcm(n_i,p)} \rangle
  \]
  and so $\zeta\in \langle \omega^{p_i} \rangle$. Conversely, if $\zeta\in
  \langle \omega^{p_i} \rangle\setminus \{1\}$ and $t\in \GL(V)$ is defined
  by
  \[
  te_l=
  \begin{cases}
    \zeta e_l &\text{if $l=\bar n_{i-1}+\nu$, $\nu\in [n_i]$} \\
    e_l &\text{if $l\in [n]\setminus \{\bar n_{i-1}+1, \dots, \bar n_i\}$,}
  \end{cases} 
  \]
  then clearly $t\in N_X^p$ and $[[t]]= r_{ii}(\zeta)$.  Because
  $(C_X^p)^\r$ is generated by the reflections it contains, it follows that
  \[
  (C_X^p)^\r\cong G(r,p_{i_1},m_1) \times \dotsb \times G(r,p_{i_c}, m_c)
  \quad\text{when} \quad k=n.
  \]
This completes the proof of the proposition.
\end{proof}

Muraleedaran \cite{muraleedaran:normalizers} has investigated the structure
of the groups $C_{X_\lambda}^p$ and shown that when $(C_{X_\lambda}^p)^\r=
C_{X_\lambda}^p$, the elements $s_{ij}(\zeta)$ defined above generate a
complement to $Z_{X_\lambda}^p$ in $N_{X_\lambda}^p$.
 
The next corollary follows from the preceding equation
and~\autoref{pro:cxref}.

\begin{corollary}\label{cor:cxref}
  Suppose $X\in L(\CA(G_p))$ is in the $G_p$-orbit of $X_{\lambda,u}$, where
  $\lambda$ is a partition of $k$ for some $1\leq k\leq n$ and $u$ is an
  integer with $0\leq u<\delta_{\lambda,p}$. Say $\lambda= (i_i^{m_1},
  \dots, i_c^{m_c})$ with $i_1> \dots >i_c>0$.
  \begin{enumerate}
  \item If $k<n$, then $(C_X^p)^\r=C_X^p$.
  \item If $k=n$, then
    \[
    |C_X^p: (C_X^p)^\r| = p^{c-1} \frac {\gcd(i_1, \dots, i_c, p)}
    {\gcd(i_1, p) \dotsb \gcd(i_c, p)}.
    \]
    In particular,
    \begin{enumerate}
    \item $C_X^p$ acts on $X$ as a reflection group if and only if $p^{c-1}
      \frac {\gcd(i_1, \dots, i_c, p)} {\gcd(i_1, p) \dotsb \gcd(i_c, p)}
      =1$, and
    \item if $\lambda=(i^m)$ has only one distinct part, then
      \[
      C_X^p= (C_X^p)^\r \cong G(r, p_i, m),
      \]
      where $p_i= \lcm(i,p)/i$.
    \end{enumerate}
  \end{enumerate}
\end{corollary}

Using this corollary we can prove the following result.

\begin{corollary}\label{cor:exp}
  Suppose $X\in L(\CA(G_p))$. If $\exp(C_X^p)\subseteq \exp(G_p)$, then
  $(C_X^p)^\r=C_X^p$.
\end{corollary}

\begin{proof}
  As in the proof of the proposition, we may assume that $X=X_\lambda$ where
  $\lambda= (i_i^{m_1}, \dots, i_c^{m_c})$ is a partition of $k$ for some
  $0\leq k\leq n$. By \autoref{cor:cxref}, if $k<n$, then
  $(C_X^p)^\r=C_X^p$. Thus we may assume that $k=n$. Then $m_1i_1+\dots +
  m_ci_c=n$ and by \autoref{pro:cxref}, $(C_{X}^p)^\r \cong G(r,p_{i_1},m_1)
  \times \dotsb \times G(r,p_{i_c}, m_c)$ where $p_{i_j}=
  p/\gcd(i_j,p)$. Using \autoref{cor:cxref} again, it is enough to show that
  $c=1$.

  It is well-known that $\exp(G(r,p,n)) = \{r-1, 2r-1, \dots, (n-1)r-1,
  nr/p-1\}$ (see \cite[Appendix B.4]{orlikterao:arrangements}). First
  consider the special case when $n=p$. Then $r=qn$ and $\exp(G_n)= \{qn-1,
  2qn-1, \dots, (n-1)qn-1, qn-1\}$. Also, $m_1q \cdot \gcd(n, i_1) -1$ is an
  exponent of $(C_X^n)^\r$ and so an exponent of $G_n$. But then $m_1q\cdot
  \gcd(n, i_1) -1 = sqn-1$ for some $s\geq 1$. Thus $m_1\cdot \gcd(n, i_1) =
  qn$, and so $s=1$ and $m_1i_1=n$. It follows that $c=1$ as desired.

  Next, just suppose that $c>1$ and that there exists $j_1$ and $j_2$ such
  that $j_1\ne j_2$ and $m_{j_1}, m_{j_2}>1$. Then $r-1$ occurs as an
  exponent of $(C_X^p)^\r$ with multiplicity at least $2$. But $r-1$ is an
  exponent of $G_p$ with multiplicity greater than $1$ if and only if $p=n$,
  and if $p=n$, then $c=1$, a contradiction. Thus it cannot be the case that
  $c>1$ and there exists $j_1$ and $j_2$ such that $j_1\ne j_2$ and
  $m_{j_1}, m_{j_2}>1$. Therefore, either $c=1$, or $c>1$ and there is an
  $s$ such that $m_j=1$ for $j \neq s$.

  Finally, just suppose that $c>1$ and $s$ is such that $m_j=1$ for $j \neq
  s$. Reordering $\lambda$ if necessary we may assume that $m_j=1$ for
  $j>1$. Suppose first that $p\nmid i_j$ for some $j>1$. Then $r /p_{i_j} -1
  = r \cdot \gcd(p, i_j)/p -1< r-1$.  The only exponent of $G_p$ that could
  be less than $r-1$ is $nr/p-1$, so $nr/p-1= r/p_{i_j}-1$. But then $n/p =
  1/p_{i_j}= \gcd(p, i_j)/p$ and so $n= \gcd(p, i_j)$. Because $c>1$,
  $i_j<n$ and so $n\ne \gcd(p, i_j)$, a contradiction.  Suppose on the other
  hand that $p| i_j$ for all $j\geq 2$. Then $r-1$ is an exponent of
  $(C_X^p)^\r$, and hence of $G_p$, with multiplicity $c>1$. Again, $r-1$ is
  an exponent of $G_p$ with multiplicity greater than $1$ if and only if
  $p=n$, and if $p=n$, then $c=1$, a contradiction. It follows that it must
  be the case that $c=1$, as claimed.
\end{proof}

\subsection*{Classification of \texorpdfstring{$\mathbf{X \in
      L(\CA(G_p))}$}{xlanr} with \texorpdfstring{$\mathbf{\CA(C_X) =
      \CA(G_p)^X}$}{cx=cxr}}

In this subsection we classify all subspaces $X\in L(\CA(G_p))$ with
$\CA(C_X^p) = \CA(G_p)^X$ and show that if $\CA(C_X^p) = \CA(G_p)^X$, then
$C_X$ acts on $X$ as a reflection group.

\begin{theorem}\label{thm:im}
  Suppose $X\in L(\CA(G_p))$ is in the $G_p$-orbit of $X_{\lambda,u}$, where
  $\lambda$ is a partition of $k$ for some $0\leq k\leq n$ and $u$ is an
  integer with $0\leq u<\delta_{\lambda,p}$.
  \begin{enumerate}
  \item If $\CA(C_X^p) = \CA(G_p)^X$, then $\lambda= (i^m)$ has only one
    distinct part.
  \item If $\lambda=(i^m)$, then $\CA(C_X^p) = \CA(G_p)^X$ unless $p=r$,
    $k=n$, $\gcd(r,i)=1$, and $i>1$.
  \end{enumerate}
\end{theorem}

\begin{proof}
  As in the preceding subsection, we assume without loss of generality that
  $X=X_\lambda$ and that $\lambda=(n_1, \dots, n_a) = (i_1^{m_1}, \dots,
  i_c^{m_c})$, where $n_1\geq \dots \geq n_a>0$ and
  $i_1>\dots>i_c>0$. Recall that $\dim X =a$.

  The restricted arrangements $\CA(G_p)^X$ have been computed by Orlik,
  Solomon, and Terao~\cite[\S6.4] {orlikterao:arrangements}. Let
  $\psi_1= |\{\,i\in [a]\mid n_i>1 \,\}|$ denote the number of parts of
  $\lambda$ greater than $1$. Then
  \begin{equation}
    \label{eq:AX}
    \CA(G_p)^X=
    \begin{cases}
      \CA_a^a &\text{if $p<r$ or $k<n$} \\
      \CA_a^{\psi_1} &\text{if $p=r$ and $k=n$.}
    \end{cases}
  \end{equation}
  
  To prove the first statement, suppose that $\CA(C_X^p) =
  \CA(G_p)^X$. There are two cases depending on $k$, $n$, and $r$.

  One case is when $p<r$ or $k<n$. Then
  \[
  \CA(G_p)^X= \CA_a^a.
  \]
  By~\autoref{pro:cxref}, $C_{X}^\r \cong G(r,\hat p_1,m_1) \times \dotsb
  \times G(r,\hat p_c, m_c)$, where for $1\leq l\leq c$, $\hat p_l=1$ if
  $k<n$ and $\hat p_l= \lcm(i_l,p)/i_l$ if $k=n$. Notice that
  $\lcm(i_l,p)/i_l<r$, because if $k=n$, then $p<r$. Thus $\hat p_l<r$ for
  $1\leq l\leq c$ and so
  \[
  \CA(C_X^p) = \CA_{m_1}^{m_1} \times \dotsb \times \CA_{m_c}^{m_c}.
  \]
  Since $\CA(C_X^p) = \CA(G_p)^X$, it must be that $c=1$ and
  $\lambda=(i^m)$, where $im=k$ and $m=\dim X=a$.

  The second case is when $p=r$ and $k=n$. Then $\CA(G_r)^X=
  \CA_a^{\psi_1}$, where $\psi_1$ is as above. Because $\CA(G_r)^X=
  \CA(C_X^r)$ is a reflection arrangement and $\CA_a^{\psi_1}$ is a
  reflection arrangement if and only if $\psi_1=0$ or $\psi_1=a$, we have
  that either
  \[
  \CA(G_r)^X= \CA_a^{0}\quad\text{or}\quad \CA(G_r)^X= \CA_a^{a}.
  \] 
  If $\CA(G_p)^X= \CA_a^{0}$, then all parts of $\lambda$ are equal to $1$,
  so $\lambda=(1^n)$. If instead $\CA(G_p)^X= \CA_a^{a}$, then using
  \autoref{pro:cxref} again we have $C_{X}^\r \cong G(r, p_{i_1},m_1) \times
  \dotsb \times G(r,p_{i_c}, m_c)$, where for $1\leq l\leq c$, $p_{i_l}
  =\lcm(i_l,r)/i_l= r/\gcd( i_l,r)$. Therefore,
  \[
  \CA(C_X^r) = \CA_{m_1}^{m_1'} \times \dotsb \times \CA_{m_c}^{m_c'},
  \]
  where $m_l'= 0$ if $\gcd(i_l, r)= 1$ and $m_l'= m_i$ if $\gcd(i_l, r)>
  1$. Finally, because $\CA^a_a= \CA(G_r)^X =\CA(C_X^r)$, it must be that
  $c=1$ and $\lambda =(i^m)$, where now $i>1$ and $\gcd(i,r)>1$.

  To prove the second statement, suppose that $\lambda= (i^m)$ has only one
  distinct part. Suppose also that $k\ne n$, or $p\ne r$, or
  $\gcd(p,i)\ne1$. Then $i\ne 1$ and so by \autoref{eq:AX} and
  \autoref{pro:cxref} $\CA(C_X^p) =\CA^m_m= \CA(G_p)^X$ as
  claimed. Otherwise $k=n$, $p=r$ and $\gcd(p,i)=1$. If $i=1$, then $X=V$
  and $C_X^p= G_p$, and so $\CA(C_X^p) =\CA= \CA(G_p)^X$ as
  claimed. Finally, if $i>1$, then by \autoref{pro:cxref} $\CA(C_X^p)=
  \CA_m^0$ and by \autoref{eq:AX} $\CA(G_p)^X =\CA_m^m$, so $\CA(C_X^p)
  \ne\CA(G_p)^X$.
\end{proof}

The next corollary follows from the theorem and~\autoref{cor:cxref}.

\begin{corollary}\label{cor:ref}
  Suppose $X\in L(\CA(G_p))$ and $\CA(C_X^p) = \CA(G_p)^X$. Then $C_X^p$
  acts on $X$ as a reflection group.
\end{corollary}

\subsection*{Proof of \autoref{thm:main} for \texorpdfstring{$
    \mathbf{ G(r,p,n)}$}{G(r,p,n)}}\label{sec:proofmain}

Suppose $X\in L(\CA(G_p))$ is such that $\CA(C_X^p) = \CA(G_p)^X$ and
$\exp(C_X^p) \subseteq \exp(G_p)$. We need to show that the restriction map
$\rho_X\colon \BBC[V]^G\to \BBC[X]^{C_X^p}$ is surjective. By the assumption
that $\CA(C_X^p) = \CA(G_p)^X$, \autoref{thm:im}, and \autoref{pro:redorb},
we may assume that $X= X_\lambda$ where $\lambda=(i^m)$. It then follows
from \autoref{cor:cxref} that $(C_X^p)^\r =C_X^p$.

If $m=n$, then $i=1$ and so $X=V$, $C_X^p=G_p$, and $\rho_X$ is the identity
map. Thus, in this case $\CA(C_X^p) = \CA(G_p)^X$, $\exp(C_X^p) \subseteq
\exp(G_p)$, and $\rho_X$ is surjective. In the rest of the proof we assume
that $m<n$ and so $X$ is a proper subspace of $V$.

Next, it follows from~\autoref{pro:cxref} that $C_X^p \cong G(r,\hat p, m)$,
where $\hat p=1$ if $im<n$ and $\hat p= p/\gcd(i,p)$ if
$im=n$. Equivalently,
\[
C_X^p \cong
\begin{cases}
  G(r,1,m) &\text{if $im<n$ or if $im=n$ and $p|i$} \\
  G(r, p/\gcd(i,p), m) &\text{if $im=n$ and $\gcd(i,p)<p$.}
\end{cases}
\]
By \cite[Table B.1]{orlikterao:arrangements} we have
\begin{itemize}
\item $\exp(G_p)= \{r-1, \dots, (n-1)r-1, nq -1\}$,
\item $\exp(C_X^p)= \{ r-1, \dots, (m-1)r-1, mr -1\}$ if $im<n$, or if
  $im=n$ and $p|i$, and
\item $\exp(C_X^p)= \{ r-1, \dots, (m-1)r-1, m \hat q -1\}$, where $\hat q=
  r\cdot \gcd(i,p)/ p$ if $im=n$ and $\gcd(i,p)<p$.
\end{itemize}

Because $m<n$, if $im<n$, or if $im=n$ and $p|i$, then the assumption that
$\exp(C_X^p) \subseteq \exp(G_p)$ is superfluous. On the other hand, if
$im=n$ and $\gcd(i,p)<p$, then $m\hat q-1\in \exp(G_p)$, and so either
$m\hat q = lr$ for some $l$ with $m\leq l\leq n-1$ or $m\hat q = nq$. Just
suppose that $m\hat q = lr$ for some $l$ with $m\leq l\leq n-1$. Then $m
\cdot \gcd(i,p) /p =l \geq m$, so $\gcd(i,p) \geq p$, so $p|i$,
contradicting the assumption that $\gcd(i,p)<p$. Therefore, it must be the
case that $m\hat q = nq$. Then $m\cdot \gcd(i,p) =n$ and $n=im$, so
$\gcd(i,p) =i$. Thus, $i|p$.

Summarizing the preceding discussion, there are two cases. Either
\begin{enumerate}
\item $im<n$, or $im=n$ and $p|i$, in which case $C_X^p\cong G(r,1,m)$ and
  $\exp(C_X^p)= \{ r-1, \dots, (m-1)r-1, mr -1\}$, or
\item $im=n$ and $i|p$, in which case $C_X^p\cong G(r,p/i,m)$ and
  $\exp(C_X^p)= \{ r-1, \dots, (m-1)r-1, nq -1\}$.
\end{enumerate}

Now consider the restriction map $\rhoXtilde\colon \BBC[V]\to
\BBC[X]$. Recall that $\{x_1, \dots, x_n\}$ is the basis of $V^*$ dual to
the basis $\{e_1, \dots, e_n\}$ of $V$ and let $\{y_1, \dots, y_m\}$ be the
basis of $X^*$ dual to $\CB_\lambda$. Then $\BBC[V] = \BBC[x_1, \dots,
x_n]$, $\BBC[X]= \BBC[y_1, \dots, y_m]$,
\begin{align*}
  \rho_X(x_1) &= \dotsm = \rho_X(x_i) = y_1, \\
  \rho_X(x_{i+1}) &= \dotsm = \rho_X(x_{2i}) = y_2, \\
  & \hskip 4.2ex \vdots \\ 
  \rho_X(x_{(m-1)i+1}) &= \dotsm = \rho_X(x_{mi}) = y_m,\\
  \intertext{and} \rho_X(x_{mi+1}) &= \dotsm = \rho_X(x_n) = 0.
\end{align*}
For $l\geq 1$ let 
\[
f_l= (x_1^r)^l+ \dots +(x_n^r)^l \quad \text{and}\quad \bar f_l= (y_1^r)^l+
\dots +(y_m^r)^l
\]
denote the $l\th$ power sums in $x_1^r, \dots, x_n^r$ and $y_1^r, \dots,
y_m^r$, respectively. Then
\[
\BBC[V]^{G_p}= \BBC[f_1, \dots, f_{n-1}, (x_1\dotsm x_n)^{q}],
\]
$\rhoXtilde( f_l) = i \bar f_l$ for all $l\geq 1$, and
\[
\rhoXtilde( (x_1\dotsm x_n)^{q})=
\begin{cases}
  0&\text{if $im<n$} \\
  (y_1\dotsm y_m)^{iq} &\text{if $im=n$.}
\end{cases}
\]
If $im<n$, or if $im=n$ and $p|i$, then
\[
\BBC[X]^{C_X^p}= \BBC[\bar f_1, \dots, \bar f_m] = \BBC[\rho_X(f_1), \dots,
\rho_X(f_m)],
\]
and so $\rho_X$ is surjective in this case.

If $im=n$ and $i|p$, then
\begin{align*}
  \BBC[X]^{C_X^p} & = \BBC[\bar f_1, \dots, \bar f_{m-1}, (y_1\dotsm y_m)^{\hat
    q}] \\ & =\BBC[\bar f_1, \dots, \bar f_{m-1}, (y_1\dotsm y_m)^{iq}] \\ & =
  \BBC[\rho_X(f_1), \dots, \rho_X(f_{m-1}), \rho_X( (x_1\dotsm x_n)^{q}) ],
\end{align*}
and so $\rho_X$ is surjective in this case as well. This completes the proof
of \autoref{thm:main} when $G=G(r,p,n)$.

\section{The exceptional irreducible unitary reflection
  groups}\label{sec:exc}

In this section we suppose that $G$ is one of the thirty-four exceptional
unitary reflection groups labeled $G_4$, \dots, $G_{37}$ and we classify the
subspaces $X\in L(\CA(G))$ such that $\CA(C_X) = \CA(G)^X$, we classify the
subspaces $X\in L(\CA(G))$ such that $C_X$ acts on $X$ as a reflection
group, and we complete the proof of~\autoref{thm:main} for these groups.

\subsection*{Classification of \texorpdfstring{$\mathbf{X}$}{xlanr} with
  \texorpdfstring{$\mathbf{\CA(C_X) = \CA(G)^X}$}{acx=ax},
  \texorpdfstring{$\mathbf{C_X = C_X^{\operatorname{\mathbf
          {ref}}}}$}{cx=cxr}, or \texorpdfstring{$\mathbf{\operatorname{
        \mathbf {exp}}(C_X) \subseteq \operatorname{\mathbf {exp}}
      (G)}$}{exp}} \label{ssec:exref}

First, as in \autoref{ssec:red}, it is enough to consider one representative
from each orbit of $G$ on $L(\CA(G))$. For the exceptional unitary
reflection groups, orbit representatives for the action of $G$ on
$L(\CA(G))$ are given in \cite[Appendix C]{orlikterao:arrangements}. We use
the labeling in those tables. Specifically, it is shown in \cite[Lemma
6.88]{orlikterao:arrangements} that for $X,Y\in L(\CA(G))$, $X$ and $Y$ are
in the same $G$-orbit if and only if $Z_X$ and $Z_Y$ are conjugate, and that
in most cases the reflection type of $Z_X$ uniquely determines $X$. Thus,
the reflection type of the subgroups $Z_X$ is used to index the orbit
containing $X$. When there are two orbits whose pointwise stabilizers have
the same reflection type, we label the orbits as in \cite[Appendix
C]{orlikterao:arrangements} with $'$ and $''$. For example, in the group
$G_{27}$ there are two orbits whose pointwise stabilizer has reflection type
$A_2$. These are denoted by $A_2'$ and $A_2''$.

Next, if $X\in L(\CA(G))$, then whether or not $\CA(C_X) = \CA(G)^X$ and
whether or not $C_X = C_X^\r$ can be determined from knowledge of the
reflection type of $C_X^\r$ as we now describe.

It was shown in the proof of \autoref{thm:surj} that $\CA(C_X) \subseteq
\CA(G)^X$, so equality holds if and only if $|\CA(C_X)|=
|\CA(G)^X|$. Because $\CA(C_X)$ is a reflection arrangement, by
\cite[Corollary 6.63]{orlikterao:arrangements} $|\CA(C_X)|$ is the sum of
the coexponents of $C_X^\r$. The coexponents of the irreducible unitary
reflection groups are given in \cite[Appendix B.4]{orlikterao:arrangements}.
By \cite[\S 6.4]{orlikterao:arrangements} and \cite{hogeroehrle:heredfree},
$\CA(G)^X$ is a free arrangement, and so by \cite[Theorem
4.23]{orlikterao:arrangements} $|\CA(G)^X|$ is the sum of the exponents of
$\CA(G)^X$. The exponents of $\CA(G)^X$ are given in the last sections of
the tables in \cite[Appendix C]{orlikterao:arrangements}. Thus, whether or
not $\CA(C_X) = \CA(G)^X$ can be determined once the reflection type of
$C_X^\r$ is known.

By definition $C_X^\r\subseteq C_X$, so equality holds if and only if
$|C_X^\r|= |C_X|$. Because $C_X^\r$ is a reflection group, $|C_X^\r|$ is the
product of the degrees of $C_X^\r$. The degrees of $C_X^\r$ are obtained
from the exponents by adding $1$, and the exponents of the irreducible
unitary reflection groups are given in \cite[Appendix
B.4]{orlikterao:arrangements}. For $|C_X|$, as described in
\cite[\S6.4]{orlikterao:arrangements}, in the tables in \cite[Appendix
C]{orlikterao:arrangements}, the size of the orbit $G\cdot X$ is the entry
in the column indexed by the reflection type of $Z_X$ in the first row. The
size of $C_X$ can then be computed using the orbit-stabilizer formula:
$|G\cdot X|= |G|/ |N_X| =|G|/ (|C_X| |Z_X|)$. Thus, whether or not $C_X =
C_X^\r$ can be determined once the reflection type of $C_X^\r$ is known.

For example, in the group $G_{34}$, if the reflection type of $Z_X$ is
$A_1A_2$, then $C_X^\r$ is isomorphic as a reflection group to the product
$C_3\times G(3,1,2)$, where $C_3$ is a cyclic group of order three acting on
a one-dimensional vector space.  Using the tables in
\cite{orlikterao:arrangements} we see that the multiset of coexponents of
$C_3 \times G(3,1,2)$ is $\{1,1,4\}$ and the multiset of exponents of
$\CA(G)^X$ is $\{1, 13, 16\}$, so $\CA(C_X) \ne \CA(G)^X$. Next, the
multiset of exponents of $C_3\times G(3,1,2)$ is $\{2, 2, 5\}$, so
$|C_X^\r|= 3\cdot 3\cdot 6=54$. On the other hand, the size of the orbit
indexed by $A_1A_2$ is $30240$, $\exp(G)= \{ 5, 11, 17, 23, 29, 41\}$, and
$\exp(Z_X) =\{1,1,2\}$, so $|C_X| = (6\cdot 12\cdot 18\cdot 24\cdot 30\cdot
42) / (30240\cdot 12)= 108$ and $C_X^\r \ne C_X$. Note also that
$\exp(C_X)\not \subseteq \exp(G)$.

It remains to compute the reflection type of $C_X^\r$ for each orbit. There
are three cases when the reflection type of $C_X^\r$ is easily determined.
First, if $X=\{0\}$, then $C_X=C_X^\r$ is the trivial group. Second, if
$\dim X = 1$, then since $C_X$ is a finite subgroup of the multiplicative
group of non-zero complex numbers, it is a cyclic group, and so clearly
$C_X=C_X^\r$. Third, if $X=V$, then $C_X=C_X^\r=G$.

The reflection types of the groups $C_X^\r$ for the exceptional Coxeter
groups were computed by Howlett~\cite{howlett:normalizers}.

The reflection types of the groups $C_X^\r$ for the exceptional,
non-Coxeter, unitary reflection groups when $1<\dim X<\dim V$ were
determined using {\sf GAP}4 (release 4.12) with the aid of the {\tt
  SmallGroups} library as follows.  First, a representative $X$ for each
orbit was chosen and then the elements $N_X$ were computed as matrices of
linear transformations of $X$. This information yielded generators of
$C_X^\r$. The reflection type of $C_X^\r$ was then determined using the {\tt
  SmallGroups} library in {\sf GAP}4, except for the case of the orbit of
type $A_1$ in the lattice of $G_{34}$, where the type of $C_X^\r$ was easy
to determine as it is an irreducible unitary reflection group of rank $5$.

The reflection types of the subgroups $C_X^\r$ for all pairs $(G,X)$, where
$G$ is an exceptional unitary reflection group with rank three or more and
$X$ is not equal $\{0\}$ or $V$ are given in
\autoref{tab:3}--\autoref{tab:e8} in the Appendix. We have included
Howlett's results for exceptional Coxeter groups, because the tables in
\cite{howlett:normalizers} contain several omissions. The columns labeled by
$\CA$ in the tables indicate whether or not $\CA(C_X) = \CA(G)^X$; the
columns labeled by $C_X$ indicate whether or not $C_X = C_X^\r$, and the
columns labeled by $\exp$ indicate whether or not $\exp(C_X) \subseteq
\exp(G)$.  For example, for the group $G_{34}$ and the orbit with pointwise
stabilizer of type $B_3$, $\CA(C_X) \ne \CA(G)^X$, $C_X^\r\ne C_X$, and
$\exp(C_X) \not \subseteq \exp(G)$.

The following proposition may be deduced from the tables.

\begin{proposition}\label{pro:excref}
  Suppose $G$ is an irreducible, exceptional unitary reflection group and
  $X\in L(\CA(G))$. If $\CA(C_X) = \CA(G)^X$ or if $\exp(C_X)\subseteq
  \exp(G)$, then $C_X^\r=C_X$.
\end{proposition}

\begin{remark}\label{rem:table}
  It is possible to have $\CA(C_X) = \CA(G)^X$ and $\exp(C_X)\subseteq
  \exp(G)$, for example the orbit indexed by $G(3,3,3)$ in $G_{34}$. It is
  also possible to have $\CA(C_X)= \CA(G)^X$ and $\exp(C_X)\not \subseteq
  \exp(G)$, for example the orbit indexed by $D_4$ in $E_6$. Similarly, it
  is possible to have $\CA(C_X) \ne \CA(G)^X$, $C_X^\r=C_X$, and
  $\exp(C_X)\subseteq \exp(G)$, for example the orbit indexed by $G(3,3,4)$
  in $G_{34}$, and it is also possible to have $\CA(C_X) \ne \CA(G)^X$,
  $C_X^\r=C_X$, and $\exp(C_X)\not \subseteq \exp(G)$, for example the orbit
  indexed by $A_1^3$ in $G_{34}$.
\end{remark}

\subsection*{Proof of \autoref{thm:main} for the exceptional unitary
  reflection groups}\label{ssec:proofmainex}

Suppose $G$ is an exceptional unitary reflection group and $X\in L(\CA(G))$
is such that $\CA(C_X) = \CA(G)^X$ and $\exp(C_X) \subseteq \exp(G)$. It
remains to show that the restriction map $\rho_X\colon \BBC[V]^G\to
\BBC[X]^{C_X^p}$ is surjective.

First, if $G$ is a Coxeter group, $\dim X=1$, and both conditions hold, then
$C_X$ is of type $A_1$ and acts as $-1$ on $X$. The canonical bilinear form
defined on $V$ in \cite[Ch.~VI \S1.1]{bourbaki:groupes} gives a non-zero,
homogeneous polynomial of degree two in $\BBC[V]^G$, say $f_2$. It is
straightforward to check that $\rho_X(f_2)\ne0$ (see \cite[\S4]
{douglassroehrle:invariants}) and it follows that $\rho_X$ is surjective in
these cases.

Finally, suppose that both conditions hold and that either $G$ is not a
Coxeter group, or that $G$ is a Coxeter group and $1<\dim X<\dim V$. All
pairs $(G,X)$ such that $\CA(C_X) = \CA(G)^X$ and
$\exp(C_X) \subseteq \exp(G)$ are highlighted in
\autoref{tab:3}--\autoref{tab:e8}. The fact that $\rho_X$ is surjective was
checked directly in every case by implementing the following argument using
a pre-packaged version of {\sf GAP}3 (release 4.4) \cite{gap3}, with the
packages {\sf CHEVIE} (version 4) \cite{chevie}, for functionality on
unitary reflection groups and {\sf VKCURVE} (version 1.2) for functionality
on multivariate polynomials, provided by J.~Michel
\cite{michel:development}.

\begin{enumerate}
\item Choose a basis $\{ x_1, \dots x_{n}\}$ of $V^*$ such that the
  restrictions of $x_1, \dots, x_a$ to $X$ are a basis of $X^*$ and the
  restrictions of $x_{a+1}, \dots, x_{n}$ are a basis of $(X^\perp)^*$
  (recall that $V$ is endowed with a $G$-invariant hermitian form). Then the
  restriction mapping $\BBC[V]\to \BBC[X]$ is given by evaluating $x_j$ at
  zero for $a+1\leq j\leq n$.
\item Let $\{f_1, \dots, f_n\}$ be the basic invariants of $G$ obtained from
  the {\sf CHEVIE} package. If the rank of $G$ is two, restrict $f_1$ and
  $f_2$ to $X$ to obtain invariants $\rho_X(f_1)$ and $\rho_X(f_2)$ of
  $C_X$. If the rank of $G$ is greater than $2$, then in all cases the
  multisets of exponents of $C_X$ and $G$ are actually sets. Choose the
  numbering so that $\deg f_1, \dots, \deg f_a$ are the degrees of
  $C_X$. Restrict $f_1$, \dots, $f_a$ to $X$ to obtain invariants
  $\rho_X(f_1)$, $\rho_X(f_2)$, \dots, $\rho_X(f_a)$ of $C_X$.
\item If the rank of $G$ is greater than $2$, compute the determinant of the
  Jacobian of $\rho_X(f_1)$, $\rho_X(f_2)$, \dots, $\rho_X(f_a)$ (with
  respect to the basis $x_1, \dots, x_a$ of $X^*$).
\end{enumerate}
If the rank of $G$ is $2$, then it turns out that at least one of
$\rho_X(f_1)$ or $\rho_X(f_2)$ is non-zero, so $\rho_X$ is surjective in
these cases. If the rank of $G$ is greater than $2$, then the Jacobian
determinant is non-zero in all cases, and so it follows from
\cite[Proposition~2.3]{springer:regular} that $\BBC[X]^{C_X}= \BBC[
\rho_X(f_1), \rho_X(f_2), \dots, \rho_X(f_a)]$.  Therefore, $\rho_X$ is
surjective in these cases as well. This completes the proof of
\autoref{thm:main}.

\section*{Appendix}\label{sec:app}

Results of the computations for exceptional groups with rank greater than
$2$ used in the proof of \autoref{thm:main} are presented in the tables
below.
\begin{itemize}
\item Orbits in $\CA(G)$ are labeled by the reflection type of the pointwise
  stabilizer of a subspace in the orbit. The order of the rows is the same
  as in \cite[Appendix C]{orlikterao:arrangements}.
\item $C_k$ denotes a cyclic group of order $k$ and we have written
  $G_{r,p,n}$ instead of $G(r,p,n)$ to save (a little) space.
\item In the columns labeled $\CA$, $Y$ indicates that $\CA(C_X) =\CA(G)^X$.
\item In the columns labeled $C_X$, $Y$ indicates that $C_X^\r=C_X$.
\item In the columns labeled $\exp$, $Y$ indicates that $\exp(C_X) \subseteq
  \exp(G)$.
\item If $X\in L(\CA)$ is such that $\CA^X= \CA(C_X^\r)$ and
  $\exp(C_X)\subseteq \exp(G)$, then the row indexed by the reflection type
  of $Z_X$ is highlighted. As explained in \autoref{ssec:proofmainex}, the
  map $\rho_X$ was checked to be surjective in all these cases.
\end{itemize}


{\tiny 
\captionsetup{font=scriptsize, skip=3pt}
\renewcommand{\arraystretch}{1.1}
  \centering
  \begin{minipage}[t]{.495\linewidth}
    \centering
    \captionof{table}{Exceptional groups of rank $3$} \label{tab:3}
    \begin{tabular} {>{$}c<{$} |%
        >{$}c<{$} >{$}c<{$} >{$}c<{$} >{$}c<{$} >{$}c<{$} }
      \addlinespace \toprule
      \addlinespace%
      G & Z_X & C_X^\r & \CA & C_X & \exp\\ \addlinespace \midrule
      \rowcolor{shade}
      G_{23} & A_0 & H_3 & Y & Y & Y\\
      (H_3) & A_1 & A_1^2 & N & Y & N\\
      \rowcolor{shade}
      & A_1^2 & A_1 & Y & Y & Y\\
      \rowcolor{shade}
      & A_2 & A_1 & Y & Y & Y\\
      \rowcolor{shade}
      & I_2(5) & A_1 & Y & Y & Y\\
      \rowcolor{shade}
      & H_3 & A_0 & Y & Y & Y\\
      \midrule
      \rowcolor{shade}
      G_{24} & A_0 & G_{24} & Y & Y & Y\\
      & A_1 & B_2 & N & Y & N\\
      & A_2 & A_1 & Y & Y & N\\
      & B_2 & A_1 & Y & Y & N\\
      \rowcolor{shade}
      & G_{24} & A_0 & Y & Y & Y\\
      \midrule
      \rowcolor{shade}
      G_{25} & A_0 & G_{25} & Y & Y & Y\\
      & C_3 & G_{3,1,2} & Y & Y & N\\
      \rowcolor{shade}
      & C_3^2 & C_6 & Y & Y & Y\\
      & G_4 & C_3 &  Y & Y & N\\
      \rowcolor{shade}
      & G_{25} & A_0  & Y & Y & Y\\
      \midrule 
      \rowcolor{shade}
      G_{26} & A_0 & G_{26} & Y & Y & Y\\ 
      \rowcolor{shade}
      & A_1 & G_5 & Y & Y & Y\\
      & C_3 & G_{6,2,2} & Y & Y & N\\
      \rowcolor{shade}
      & A_1 C_3 & C_6 & Y & Y & Y\\
      \rowcolor{shade}
      & G_4 & C_6 & Y & Y & Y\\
      \rowcolor{shade}
      & G_{3,1,2} & C_6 & Y & Y & Y\\ 
      \rowcolor{shade}
      & G_{26} & A_0 & Y & Y & Y\\
      \midrule
      \rowcolor{shade}
      G_{27} & A_0 & G_{27} & Y & Y & Y\\
      & A_1 & B_2 & N & N & N\\
      \rowcolor{shade}
      & A_2' & C_6 & Y & Y & Y\\
      \rowcolor{shade}
      & A_2'' & C_6 & Y & Y & Y\\
      \rowcolor{shade}
      & B_2 & C_6 & Y & Y & Y\\
      \rowcolor{shade}
      & I_2(5) & C_6 & Y & Y & Y\\
      \rowcolor{shade}
      & G_{27} & A_0 & Y & Y & Y\\
      \bottomrule
    \end{tabular}
  \end{minipage}
  \begin{minipage}[t]{.495\linewidth}
    \centering
    \captionof{table}{Exceptional groups of rank $4$} \label{tab:4}
    \begin{tabular} {%
        >{$}c<{$} | >{$}c<{$} >{$}c<{$} >{$}c<{$} >{$}c<{$} >{$}c<{$}}
      \addlinespace  \toprule \addlinespace%
      G & Z_X & C_X^\r & \CA & C_X & \exp\\ \addlinespace \midrule
      \rowcolor{shade}
      G_{28} & A_0 & F_4 & Y & Y & Y\\
      (F_4) & A_1 & B_3 & N & Y & N\\
      & \tilde A_1 & B_3 & N & Y & N\\
      & A_1 \tilde A_1 & A_1^2 & N & Y & N\\ 
      \rowcolor{shade}
      & A_2 & G_2 & Y & Y & Y\\
      \rowcolor{shade}
      & \tilde A_2 & G_2 & Y & Y & Y\\
      & B_2 & B_2 & Y & Y & N\\
      \rowcolor{shade}
      & C_3\footnote{A Coxeter group of type $C_3$, not a cyclic group.} 
      & A_1 & Y & Y & Y\\
      \rowcolor{shade}
      & B_3 & A_1 & Y & Y & Y\\
      \rowcolor{shade}
      & A_1 \tilde A_2 & A_1 & Y & Y & Y\\
      \rowcolor{shade}
      & \tilde A_1A_2 & A_1 & Y & Y & Y\\
      \rowcolor{shade}
      & F_4 & A_0 & Y & Y & Y\\
      \midrule
      \rowcolor{shade}
      G_{29} & A_0 & G_{29} & Y & Y & Y\\
      & A_1 & B_3 & N & N & N\\
      & A_1^2 & G_{4,2,2} & N & Y & N\\
      & A_2 & A_1^2 & N & N & N\\
      \rowcolor{shade}
      & B_2 & G_{4,1,2} & Y & Y & Y\\
      \rowcolor{shade}
      & A_1A_2 & C_4 & Y & Y & Y\\
      \rowcolor{shade}
      & A_3' & C_4 & Y & Y & Y\\
      \rowcolor{shade}
      & A_3'' & C_4 & Y & Y & Y\\
      \rowcolor{shade}
      & B_3 & C_4 & Y & Y & Y\\
      \rowcolor{shade}
      & G_{4,4,3} & C_4 & Y & Y & Y\\
      \rowcolor{shade}
      & G_{29} & A_0 & Y & Y & Y\\
      \midrule
      \rowcolor{shade}
      G_{30} & A_0 & H_4 & Y & Y & Y\\
      (H_4) & A_1 & G_{23} & N & Y & N\\
      & A_1^2 & B_2 & N & Y & N\\
      & A_2 & G_2 & N & Y & N\\
      & I_2(5) & C_{20} & N & Y & N\\
      \rowcolor{shade}
      & A_1A_2 & A_1 & Y & Y & Y\\
      \rowcolor{shade}
      & A_1I_2(5) & A_1 & Y & Y & Y\\
      \rowcolor{shade}
      & A_3 & A_1 & Y & Y & Y\\
      \rowcolor{shade}
      & H_3 & A_1 & Y & Y & Y\\
      \rowcolor{shade}
      & H_4 & A_0 & Y & Y & Y\\
      \midrule
      \rowcolor{shade}
      G_{31} & A_0 & G_{31} & Y & Y & Y\\
      & A_1 & G_{4,1,3} & N & Y & N\\
      & A_1^2 & G_{4,1,2} & N & Y & N\\
      & A_2 & G_2 & N & N & N\\
      \rowcolor{shade}
      & G_{4,2,2} & G_8 & Y & Y & Y\\
      & A_1A_2 & C_4 & Y & Y & N\\
      & A_3 & C_4 & Y & Y & N\\
      & G_{4,2,3} & C_4 & Y & Y & N\\
      \rowcolor{shade}
      & G_{31} & A_0 & Y & Y & Y\\
      \midrule
      \rowcolor{shade}
      G_{32} & A_0 & G_{32} & Y & Y & Y\\
      & C_3 & G_{26} & Y & Y & N\\
      & G_4 & G_5 & Y & Y & N\\
      & C_3^2 & G_{6,1,2} & Y & Y & N\\
      & C_3G_4 & C_6 & Y & Y & N\\
      & G_{25} &C_6 & Y & Y & N\\
      \rowcolor{shade}
      & G_{32} & A_0 & Y & Y & Y\\
      \bottomrule
    \end{tabular}
  \end{minipage}
} 

\newpage

{ \scriptsize
  \captionsetup{font=scriptsize, skip=5pt}
  \renewcommand{\arraystretch}{1.1}
  \centering
  \begin{minipage}[t]{.49\linewidth}
    \centering
    \captionof{table}{The exceptional group $G_{33}$} \label{tab:g33}
    \begin{tabular} {>{$}c<{$} |%
        >{$}c<{$} >{$}c<{$} >{$}c<{$} >{$}c<{$}}
      \toprule   
      Z_X & C_X^\r & \CA & C_X & \exp\\
      \midrule
      \rowcolor{shade}
      A_0 & G_{33} & Y & Y & Y\\
      A_1 & D_4 & N & N & N\\
      A_1^2 & B_3 & N & Y & N\\
      A_2 & G_{3,1,2} & N & N & N\\
      \rowcolor{shade}
      A_1^3 & G_{6,3,2}& Y & Y & Y\\
      A_1A_2 & C_3 & N & N & N\\
      A_3 & A_1^2 & N & Y & N\\
      \rowcolor{shade}
      G_{3,3,3} & G_4 & Y & Y & Y\\
      A_1A_3 & A_1 & Y & Y & N\\
      A_4 & A_1 & Y & Y & N\\
      \rowcolor{shade}
      D_4 & C_6 & Y & Y & Y\\
      G_{3,3,4} & A_1 & Y & Y & N\\
      \rowcolor{shade}
      G_{33} & A_0 & Y & Y & Y\\
      \bottomrule
    \end{tabular}
    \vglue 5ex
    \captionof{table}{The exceptional group $G_{34}$} \label{tab:g34}
    \begin{tabular} {%
        >{$}c<{$} | >{$}c<{$} >{$}c<{$} >{$}c<{$} >{$}c<{$}}
      \toprule
      Z_X & C_X^\r & \CA & C_X & \exp\\
      \midrule
      \rowcolor{shade}
      A_0 & G_{34}  & Y & Y & Y\\
      A_1 & G_{33} & N & N & N\\
      A_1^2 & F_4 &  N & N & N\\
      A_2 & G_{3,1,4} & N & N & N\\
      A_1^3 & G_{6,3,3} & N & Y & N\\
      A_1A_2 & C_3 G_{3,1,2} & N & N & N \\
      A_3 & B_3 & N & N & N\\
      \rowcolor{shade}
      G_{3,3,3} & G_{26} & Y & Y & Y\\
      A_1^2A_2 & G_{6,2,2} & N & Y & N\\
      A_2^2 & G_{6,2,2} & N & Y & N\\
      A_1A_3 & A_1^2 & N & N & N\\
      A_4 & A_1^2 & N & N & N\\
      \rowcolor{shade}
      D_4 & G_{6,1,2} & Y & Y & Y\\
      \rowcolor{shade}
      A_1G_{3,3,3} & G_5 & Y & Y & Y\\
      G_{3,3,4} & G_{6,2,2}& N & Y & Y \\
      \rowcolor{shade}
      A_2A_3 & C_6 & Y & Y & Y\\
      \rowcolor{shade}
      A_1A_4 & C_6 & Y & Y & Y\\
      \rowcolor{shade}
      A_5' & C_6 & Y & Y & Y\\
      \rowcolor{shade}
      A_5'' & C_6 & Y & Y & Y\\
      \rowcolor{shade}
      D_5 & C_6 & Y & Y & Y\\
      \rowcolor{shade}
      A_1G_{3,3,4} & C_6 & Y & Y & Y\\
      \rowcolor{shade}
      G_{3,3,5} & C_6 & Y & Y & Y\\
      \rowcolor{shade}
      G_{33} & C_6 & Y & Y & Y\\
      \rowcolor{shade}
      G_{34} & A_0 & Y & Y & Y\\
      \bottomrule
    \end{tabular}
  \end{minipage}
  \begin{minipage}[t]{.49\linewidth}
    \centering
    \captionof{table}{The exceptional group $G_{35}$ ($E_6$)} \label{tab:e6}
    \begin{tabular} {%
        >{$}c<{$} | >{$}c<{$} >{$}c<{$} >{$}c<{$} >{$}c<{$}}
      \toprule 
      Z_X & C_X^\r & \CA & C_X & \exp\\ 
      \midrule
      \rowcolor{shade}
      A_0 & E_6 & Y & Y & Y\\
      A_1 & A_5 & N & Y & N\\
      A_1^2 & B_3 & N & Y & N\\
      A_2 & A_2^2 & N & N & N\\
      A_1^3 & A_1A_2 & N & Y & N\\
      A_1A_2 & A_2 & N & Y & N\\
      A_3 & B_2 & N & Y & N\\
      A_1^2A_2 & A_1 & N & Y & Y\\
      \rowcolor{shade}
      A_2^2 & G_2 & Y & Y & Y\\
      A_1A_3 & A_1 & N & Y & Y\\
      A_4 & A_1 & N & Y & Y\\
      D_4 & A_2 & Y & Y & N\\
      \rowcolor{shade}
      A_1A_2^2 & A_1 & Y & Y & Y\\
      A_1A_4 & A_0 &  N & Y & N\\
      \rowcolor{shade}
      A_5 & A_1 & Y & Y & Y\\
      D_5 & A_0 & N & Y & N\\
      \rowcolor{shade}
      E_6 & A_0 & Y & Y & Y\\
      \bottomrule
    \end{tabular}
  \end{minipage}
}

\newpage

{\scriptsize 
  \captionsetup{font=scriptsize, skip=5pt}
\begin{minipage}[t]{.49\linewidth}
    \centering
    \captionof{table}{The exceptional group $G_{36}$ ($E_7$)} \label{tab:e7}
    \begin{tabular} {%
        >{$}c<{$} >{$}c<{$} >{$}c<{$} >{$}c<{$} >{$}c<{$}}
      \toprule 
      Z_X & C_X^\r & \CA & C_X & \exp \\
      \midrule
      \rowcolor{shade}
      A_0 & E_7 & Y & Y & Y\\
      A_1 & D_6 & N & Y & N\\
      A_1^2 & A_1B_4 & N & Y & N\\
      A_2 & A_5 & N & N & N\\
      \rowcolor{shade}
      (A_1^3)' & F_4 & Y & Y & Y\\
      (A_1^3)'' & A_1B_3 & N & Y & N\\
      A_1A_2 & A_3 & N & Y & N\\
      A_3 & A_1B_3 & N & Y & N\\
      A_1^4 & B_3 & N & Y & N\\
      A_1^2A_2 & A_1^3 & N & Y & N\\
      A_2^2 & A_1G_2 & N & Y & N\\
      (A_1A_3)' & B_3 & N & Y & N\\
      (A_1A_3)'' & A_1^3 & N & Y & N\\
      A_4 & A_2 & N & N & N\\
      D_4 & B_3 & Y & Y & N\\
      \rowcolor{shade}
      A_1^3A_2 & G_2 & Y & Y & Y\\
      A_1A_2^2 & A_1^2 & N & Y & N\\
      A_1^2A_3 & A_1^2 & N & Y & N\\
      A_2A_3 & A_1^2 & N & Y & N\\
      A_1A_4 & A_0 & N & N & N\\
      \rowcolor{shade}
      A_5' & G_2 & Y & Y & Y\\
      A_5'' & A_1^2 & N & Y & N\\
      A_1D_4 & B_2 & Y & Y & N\\
      D_5 & A_1^2 & N & Y & N\\
      \rowcolor{shade}
      A_1A_2A_3 & A_1 & Y & Y & Y\\
      \rowcolor{shade}
      A_2A_4 & A_1 & Y & Y & Y\\
      \rowcolor{shade}
      A_1A_5 & A_1 & Y & Y & Y\\
      \rowcolor{shade}
      A_6 & A_1 & Y & Y & Y\\
      \rowcolor{shade}
      A_1D_5 & A_1 & Y & Y & Y\\
      \rowcolor{shade}
      D_6 & A_1 & Y & Y & Y\\
      \rowcolor{shade}
      E_6 & A_1 & Y & Y & Y\\
      \rowcolor{shade}
      E_7 & A_0 & Y & Y & Y\\
      \bottomrule
    \end{tabular}    
  \end{minipage}
  \begin{minipage}[t]{.49\linewidth}
    \centering
    \captionof{table}{The exceptional group $G_{37}$ ($E_8$)} \label{tab:e8}
    \begin{tabular} {%
        >{$}c<{$} >{$}c<{$} >{$}c<{$} >{$}c<{$} >{$}c<{$}}
      \toprule 
      Z_X & C_X^\r & \CA & C_X & \exp \\
      \midrule
      \rowcolor{shade}
      A_0 & E_8 & Y & Y & Y\\
      A_1 & E_7 & N & Y & N\\
      A_1^2 & B_6 & N & Y & N\\
      A_2 & E_6 & N & N & N\\
      A_1^3 & A_1F_4 & N & Y & N\\
      A_1A_2 & A_5 & N & N & N\\
      A_3 & B_5 & N & N & N\\
      A_1^4 & B_4 & N & Y & N\\
      A_1^2A_2 & A_1B_3 & N & Y & N\\
      A_2^2 & G_2^2 & N & N & N\\
      A_1A_3 & A_1B_3 & N & Y & N\\
      A_4 & A_4 & N & N & N\\
      D_4 & F_4 & Y & Y & N\\
      A_1^3A_2 & A_1G_2 & N & Y & N\\
      A_1A_2^2 & A_1G_2 & N & Y & N\\
      A_1^2A_3  & A_1B_2 & N & Y & N\\
      A_2A_3 & A_1B_2 & N & Y & N\\
      A_1A_4 & A_2 & N & N & N\\
      A_5 & A_1G_2 & N & Y & N\\
      A_1D_4 & B_3 & N & Y & N\\
      D_5 & B_3 & N & Y & N\\
      A_1^2A_2^2 & B_2 & N & Y & N\\
      A_1A_2A_3 & A_1^2 & N & Y & N\\
      A_1^2A_4 & A_1^2 & N & Y & N\\
      A_3^2 & B_2 & N & Y & N\\
      A_2A_4 & A_1^2 & Y & Y & N\\
      A_1A_5 & A_1^2 & N & Y & N\\
      A_6 & A_1^2 & N & Y & N\\
      A_2D_4 & G_2 & Y & Y & N\\
      A_1D_5 & A_1^2 & N & Y & N\\
      D_6 & B_2 & N & Y & N\\
      E_6 & G_2 & Y & Y & N\\
      \rowcolor{shade}
      A_1A_2A_4  & A_1 & Y & Y & Y\\
      \rowcolor{shade}
      A_3A_4  & A_1 & Y & Y & Y\\
      \rowcolor{shade}
      A_1A_6  & A_1 & Y & Y & Y\\
      \rowcolor{shade}
      A_7 &  A_1 & Y & Y & Y\\
      \rowcolor{shade}
      A_2D_5  & A_1 & Y & Y & Y\\
      \rowcolor{shade}
      D_7 &  A_1 & Y & Y & Y\\
      \rowcolor{shade}
      A_1E_6  & A_1 & Y & Y & Y\\
      \rowcolor{shade}
      E_7 &  A_1 & Y & Y & Y\\
      \rowcolor{shade}
      E_8 &  A_0 & Y & Y & Y\\
      \bottomrule
    \end{tabular}
  \end{minipage}
}



\bigskip \noindent {\bf Acknowledgments:} This work was partially supported
by a grant from the Simons Foundation (Grant \#245399 to
J.M.~Douglass). J.M.~Douglass would like to acknowledge that some of this
material is based upon work supported by (while serving at) the National
Science Foundation. We acknowledge support from the DFG-priority program
SPP1489 ``Algorithmic and Experimental Methods in Algebra, Geometry, and
Number Theory.'' We are grateful to S.~Casalaina-Martin and J.~Michel for
helpful discussions.


\bibliographystyle{plain}


\end{document}